\documentclass[a4paper,10pt,final]{scrartcl}
\usepackage[latin1]{inputenc}
\usepackage{amsmath,amsthm}
\usepackage{amsfonts}
\usepackage{amssymb}
\usepackage{a4wide}
\usepackage{mathabx}
\usepackage{mathrsfs}
\usepackage[numbers,square]{natbib}
\usepackage{float}
\usepackage{xcolor}
\usepackage{bbm}
\usepackage{subfigure}
\usepackage{pgfplots}
\usepackage{tikz}
\pgfplotsset{compat=newest}
\newlength\fheight \newlength\fwidth

\usepackage{hyperref}
\hypersetup{
  	colorlinks=true,
	linkcolor=blue,
  	citecolor=red}

\newcommand{\1}{\boldsymbol{1}}

\newcommand{\E}[2]{\mathbb E_{#1}\left[#2\right]}
\def\Var{{\mathrm {Var}\;}}
\newcommand{\Prob}[2]{\mathbb P_{#1} \left[#2\right]}

\newcommand{\essinf}{\mathrm{essinf}}

\newcommand{\idmat}{\textup{id}}

\newcommand{\R}{\mathbb R}
\newcommand{\N}{\mathbb N}
\newcommand{\Z}{\mathbb Z}

\DeclareMathOperator{\cov}{Cov}
\DeclareMathOperator{\var}{Var}

\newtheorem{Theorem}{Theorem}
\newtheorem{Corollary}[Theorem]{Corollary}
\newtheorem{Lemma}[Theorem]{Lemma}
\newtheorem{Assumption}{Assumption}
\newtheorem{Remark}[Theorem]{Remark}
\newtheorem{Definition}[Theorem]{Definition}

\numberwithin{Theorem}{section}
\numberwithin{equation}{section}

\begin{document}
	
\begin{center}
	\begin{minipage}{.8\textwidth}
		\centering 
		\LARGE Bump detection in the presence of dependency: Does it ease or does it load?\\[0.5cm]
		
		\normalsize
		\textsc{Farida Enikeeva}\\[0.1cm]
		\verb+farida.enikeeva@math.univ-poitiers.fr+\\
		Laboratoire de Math\'ematiques et Applications, UMR CNRS 7348, Universit\'e de Poitiers, France \\
		and \\
		The A. A. Kharkevich Institute for Information Transmission Problems, Russian Academy of Science, Moscow, Russia\\[0.1cm]
		\textsc{Axel Munk, Markus Pohlmann}\\[0.1cm]
		\verb+munk@math.uni-goettingen.de,+\\
		\verb+markus.pohlmann@mathematik.uni-goettingen.de+\\
		Institute for Mathematical Stochastics, University of G\"ottingen\\
		and\\
		Felix Bernstein Institute for Mathematical Statistics in the Bioscience, University of G\"ottingen\\
		and\\
		Max Planck Institute for Biophysical Chemistry, G\"ottingen, Germany\\[0.1cm]
		\textsc{Frank Werner}\footnotemark[1]\\[0.1cm]
		\verb+frank.werner@mathematik.uni-wuerzburg.de+\\
		Institut f\"ur Mathematik, University of Wuerzburg, Germany
	\end{minipage}
\end{center}

\footnotetext[1]{Corresponding author}

\begin{abstract}
	We provide the asymptotic minimax detection boundary for a bump, i.e. an abrupt change, in the mean function of a stationary Gaussian process. This will be characterized in terms of the asymptotic behavior of the bump length and height as well as the dependency structure of the process. A major finding is that the asymptotic minimax detection boundary is generically determined by the value of its spectral density at zero. Finally, our asymptotic analysis is complemented by non-asymptotic results for AR($p$) processes and confirmed to serve as a good proxy for finite sample scenarios in a simulation study. Our proofs are based on laws of large numbers for non-independent and non-identically distributed arrays of random variables and the asymptotically sharp analysis of the precision matrix of the process. 
\end{abstract}

\textit{Keywords:}  minimax testing, time series, ARMA processes, change point detection, weak laws of large numbers, Toeplitz matrices \\[0.1cm]

\textit{AMS classification numbers: } Primary 62F03, 62M07, Secondary 60G15, 62H15. \\[0.3cm]

\section{Introduction}

\subsection{Model and problem statement}
In this paper we consider observations of a triangular array of Gaussian vectors, $Y=\mu_n+\xi_n$, $n \in \mathbb N$ with the coordinates
\begin{equation}\label{eq:model_dep}
Y_{i,n}=\mu_{i,n}+\xi_{i,n},\qquad \xi_n=(\xi_{1,n}, \ldots,\xi_{n,n})^T\sim \mathcal N_n \left(0, \Sigma_n\right),
\end{equation}
with a known positive definite covariance matrix $\Sigma_n \in \R^{n \times n}$, but an unknown mean vector $\mu_n=(\mu_{1,n}, \ldots,\mu_{n,n})^T \in \R^n$. We will furthermore assume that the noise $\xi_n$ in \eqref{eq:model_dep} consists of $n$ consecutive samples of a stationary process $\left(Z_t\right)_{t \in \Z}$. 

For a proper asymptotic treatment, we will assume that $\mu_n$ is obtained from equidistantly sampling a function $m_n:[0,1]\to \mathbb{R}$ at sampling points $\frac{i}{n}, i=1,\ldots, n$, i.e. $\mu_n = \left(m_n \left(\frac1n\right),\dots,m_n\left(\frac{n}{n}\right)\right)^T$. Our goal is to analyze how difficult it is to detect abrupt changes of the function $m_n$ based on the observations $Y = \left(Y_{1,n},...,Y_{n,n}\right)^T$ coming from \eqref{eq:model_dep}. Therefore, we focus on functions $m_n$ of the form

\begin{equation}\label{eq:mu_n}
m_n \left(x\right) = \begin{cases} \delta_n &\text{if }x\in I_n, \\ 0 & \text{else,} \end{cases} 
\end{equation}
i.e. $m_n$ has a bump located at the interval $I_n \subset \left[0,1\right]$ of height $\delta_n \in \mathbb R$, see also Figure \ref{fig:model} for an illustration. We assume throughout the paper that the matrix $\Sigma_n$ in \eqref{eq:model_dep} as well as the length of the bump $\lambda_n \in \left(0,1\right)$ are known, but that its amplitude $\delta_n$ and the exact position of the bump itself are unknown.

%However, we will show in Remark \ref{rem:adaptivity} that not knowing $\Delta_n$ or allowing for $\Delta_n < 0$ does not alter our findings.

To formalize the detection problem, let us introduce some notation. For an interval $I \subset \left[0,1\right]$ we use $\1_I\in\R^n$ as abbreviation for the vector with entries
\[
\1_I(i)=\begin{cases} 1 & \mbox{if } \frac{i}{n}\in I,\\
0 & \mbox{else}, \end{cases}\qquad 1 \leq i \leq n.
\]
Consequently, $\mu_n=\delta_n\1_{I_n}$ whenever $m_n$ is of the form \eqref{eq:mu_n}. Furthermore let
\[
\mathcal I := \left\{ \left[a,b\right)~\big|~ 0 \leq a < b \leq 1 \right\}
\]
be the set of all right-open intervals in $\left[0,1\right]$, and for a given length $\lambda \in \left(0,1\right)$ we introduce by
\[
\mathcal I \left(\lambda \right):= \left\{\left[a,b\right)~\big|~ 0 \leq a < b \leq 1, b-a = \lambda\right\}
\]
the set of all right-open intervals in $\left[0,1\right]$ of length $\lambda$. 

Now the problem to detect a bump of length $\lambda_n$ in the signal $\mu_n$ from \eqref{eq:model_dep} can be understood as the hypothesis testing problem
\begin{gather}
H_0^n: Y \sim \mathcal N_n \left(0,\Sigma_n\right)\nonumber \\
\mbox{against}  \label{eq:testing_problem} \\ 
H_1^n: \exists I \in \mathcal I \left(\lambda_n\right),\ \exists\delta\in\R: |\delta|\ge \Delta_n\quad \mbox{such that} \quad Y \sim \mathcal{N}_n \left(\delta\1_{I},\Sigma_n\right)\nonumber
\end{gather}
with a minimal amplitude value $\Delta_n > 0$ to ensure distinguishability of $H_0^n$ and $H_1^n$. Note that $I$ and $\delta$ in \eqref{eq:testing_problem} are allowed to depend on $n$ (as the length $\lambda_n$ and the minimal amplitude value $\Delta_n$ do), but we suppress this dependency in the following. Similarly we write $H_0$ instead of $H_0^n$ as $\Sigma_n$ is assumed to be known. {Note that we will consider the situation $\lambda_n \to 0$ as $n \to \infty$ below, corresponding to a vanishing bump, which avoids trivial cases such as $\E{H_1^n}{Y_i} = \delta > 0$ for all $1\leq i \leq n$ in \eqref{eq:testing_problem}.}

The aim of this paper is to provide insight on how the dependency structure in \eqref{eq:model_dep} encoded in terms of $\Sigma_n$ influences the detection of such a bump. More precisely, we would like to derive asymptotic conditions\footnote{Let $\left(a_n\right)_{n \in\N}$ and $\left(b_n\right)_{n \in \N}$ two sequences of positive numbers. In the following we write $a_n \sim b_n$ if $0 < \liminf_{n\to\infty} a_n / b_n \leq \limsup_{n \to \infty} a_n / b_n < \infty$, and $a_n \asymp b_n$ if $\lim_{n\to\infty} a_n / b_n = 1$.} on the minimal detectable bump amplitude $\Delta_n$ depending on $\Sigma_n$, $\lambda_n$ and $n$. To the best of our knowledge, there is no systematic understanding of this problem from the minimax point of view. We will therefore provide (asymptotic) lower and upper bounds for the amplitude of asymptotically detectable signals in the following sense (cf. \cite{i93,Ingster&Suslina:2003}). Let $\alpha,\beta\in(0,1)$ be arbitrary error levels.
\begin{description}
	\item[upper detection bound:] Whenever the bump amplitude $\Delta_n$ satisfies $\Delta_n = c \varphi_n$, $c \geq c^*$ with a constant $c^* > 0$ and a rate $\varphi_n$ depending on $n$, $\lambda_n$ and $\Sigma_n$, then \textbf{there is a sequence of tests} for \eqref{eq:testing_problem} with (asymptotic) type I error $\leq \alpha$ and (asymptotic) type II error $\leq \beta$.
	\item[lower detection bound:] Whenever the bump amplitude $\Delta_n$ satisfies $\Delta_n = c \tilde \varphi_n$, $c  \leq c_*$ with a constant $c_* > 0$ and a rate $\tilde \varphi_n$ depending on $n$, $\lambda_n$ and $\Sigma_n$, then \textbf{no sequence of tests} for \eqref{eq:testing_problem} can have type (asymptotic) I error $\leq \alpha$ and at the same time (asymptotic) type II error $\leq \beta$.
\end{description}
Precise definitions of the (asymptotic) type I and type II errors and comments on the validity of these particular notions of the detection bounds will be given in Section \ref{subsec:notations}. Note that the minimax separation rate $\varphi_n$ might depend on the prescribed significance levels $\alpha$ and $\beta$, and that the definitions become trivial if $\beta \geq 1-\alpha$, as then any standard Bernoulli experiment with success probability $\alpha$ defines a corresponding test. However, in our case neither the constants $c_*$ and $c^*$  nor the rate depend on the error levels $\alpha$ and $\beta$. That is why in the following we will always choose $\alpha=\beta \in \left(0,\frac12\right)$ and argue in Section \ref{subsec:notations} that this is sufficient.

If $\tilde \varphi_n = \varphi_n$ in the above upper and lower bound, then we speak of the \textbf{(asymptotic) minimax separation rate} $\Delta_n \sim\varphi_n$. If furthermore $c^* = c_*$, then $\Delta_n \asymp c_* \varphi_n = c^* \tilde \varphi_n$  is called the \textbf{(asymptotic) minimax detection boundary} over all possible amplitudes $\Delta_n > 0$ and positions $I \in \mathcal I \left(\lambda_n\right)$. We will provide explicit expressions for this under weak assumptions on the covariance matrix $\Sigma_n$.

% All our results hold actually true for noise $\xi_n$ in \eqref{eq:model_dep} not necessarily arising from consecutive samples of a stationary process as long as the covariance matrix $\Sigma_n$ obeys a Toeplitz structure. 
We will provide lower and upper bounds in terms of sums over diagonal blocks within $\Sigma_n$ (cf. Section \ref{sec:AR_non_asympt} and Lemmas \ref{lm:app_upper_bound} and \ref{lm:app_lower_bound}), and for the case of noise generated by subsequent samples of a stationary time series we will show that these lower and upper bounds coincide.

In case of i.i.d. observations, this is $\Sigma_n = \sigma^2\idmat_n$ in \eqref{eq:model_dep}, the minimax detection boundary is well-known and given by (see \citep{ds01,cw13,fms14})
\begin{equation}\label{eq:db}
\Delta_n \asymp \sigma\sqrt{\frac{-2 \log\lambda_n}{n\lambda_n}}.
\end{equation}
Here, and in the following, we require
\begin{equation}\label{eq:I}
\lambda_n\to 0 \quad\text{and}\quad n\lambda_n\to\infty\qquad\text{as}\qquad n\to\infty.
\end{equation}
Signals for which the left-hand side in \eqref{eq:db} is asymptotically larger than the right-hand side can be detected consistently (in the sense of an upper detection bound as described above), whereas they can not be detected consistently once the left-hand side in \eqref{eq:db} is asymptotically smaller than the right-hand side (in the sense of a lower detection bound as described above). Although \eqref{eq:db} is known for a long time when the errors are i.i.d., to the best of our knowledge, the influence of the error dependency structure on the detection boundary \eqref{eq:db} is an issue that is much less investigated systematically, although many methods to estimate such abrupt changes in the signal corrupted by serially dependent errors have been suggested (see Section \ref{subsec:related}). In this sense, this paper contributes a benchmark to such methods. Let us illustrate the effect of the dependency on \eqref{eq:db} with $\xi_n$ in \eqref{eq:model_dep} arising from an AR($1$) process with unit variance and auto-correlation coefficient $\rho$, this is $\xi_n = \left(1-\rho^2\right)^{1/2}\left(Z_1, ..., Z_n\right)^T$ where $Z_t-\rho  Z_{t-1}=\zeta_t$ with i.i.d. standard Gaussian noise $\zeta_t, t \in \mathbb Z$. In Figure \ref{fig:model} we illustrate three different situations encoded in terms of $\rho$, namely positively correlated noise ($\rho = 0.7$), independent noise ($\rho = 0$), and negatively correlated noise ($\rho = -0.7$). It seems intuitively clear that the value of $\rho$ influences the difficulty of detecting a bump substantially, and especially positively correlated noise hinders efficient detection dramatically. Compare e.g. the first plot in Fig.~\ref{fig:model}, where noise and bump appear hardly to distinguish. Furthermore, due to the positive correlation, there appear several regions which suggest a bump in signal, which is not there. In contrast, the middle and bottom plot allow for simpler identification of the bump region. Our main result makes these intuitive findings precise.

\begin{figure}[!htb]
	\setlength\fheight{3cm} \setlength\fwidth{12cm}
	\centering
	\input{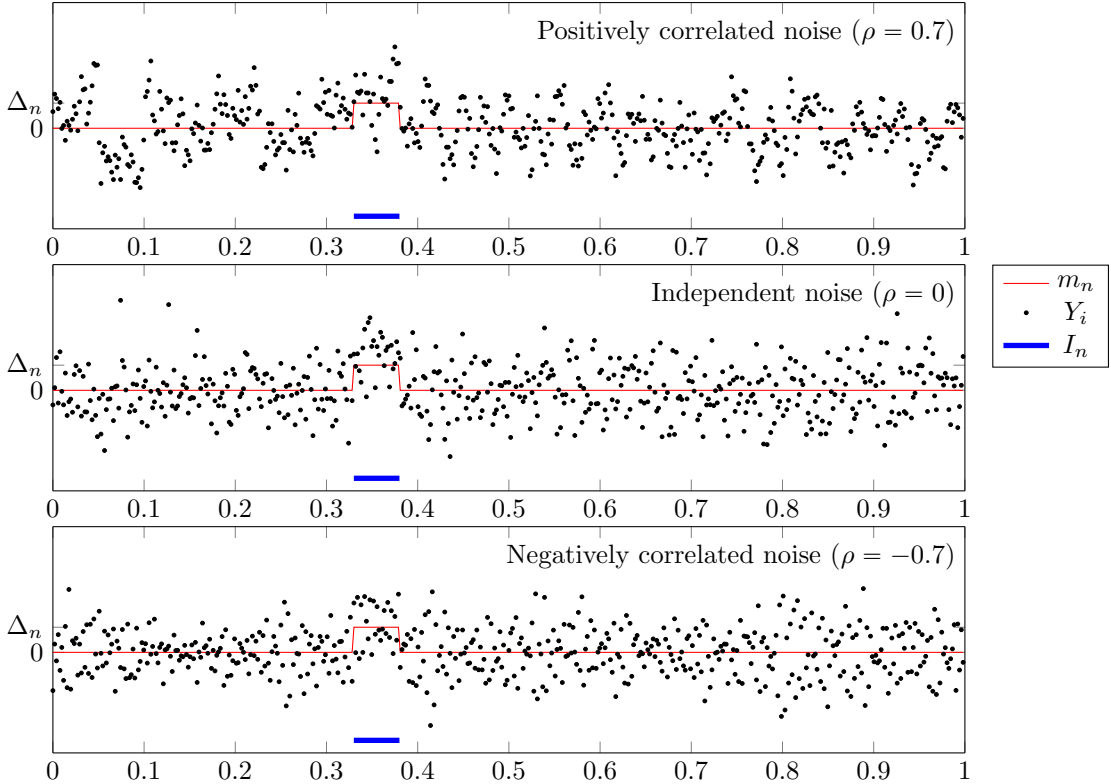}
	\caption{Model \eqref{eq:model_dep} in case of AR($1$) noise for different values of $\rho$: Data together with the function $m_n$, where the model parameters are set to be $n=512$ and $\Delta_n = 1$, $\sigma=1$.}
	\label{fig:model}
\end{figure}

\subsection{Results}

To describe our results concerning the detection boundary for serially dependent data we require some more terminology. Let the autocovariance function $\gamma_Z$ of the stationary process $\left(Z_t\right)_{t \in \mathbb Z}$ be given by  $\gamma_Z(h)=\text{Cov}\left[Z_t, Z_{t+h}\right]$ for $h\in\mathbb{Z}$. Assume that $\gamma_Z$ is square summable, then the process $Z$ has the spectral density $f_Z\in \mathbf L_2\left[-1/2,1/2\right)$ defined by 
\[
f_Z(\nu)=\sum_{h=-\infty}^{\infty} \gamma_Z(h)e^{-2\pi i h\nu}, \qquad\nu \in \left[-1/2,1/2\right).
\]
In fact, $f_Z$ can also be considered as a function on the unit sphere, i.e. one naturally has $\lim_{\nu\to 1/2}f_Z(\nu)=f_Z(-1/2)$. We will also assume that the autocovariance function is symmetric, which is equivalent to $f_Z$ being real-valued. In the following, we will omit the subscript $Z$ in the notation of the spectral density of $Z$ when it does not create ambiguities.

With this notation introduced, we will show under mild conditions that the detection boundary for the hypothesis testing problem \eqref{eq:testing_problem} is given by
\[
\Delta_n \asymp\sqrt{\frac{-2 f\left(0\right) \log\lambda_n}{n\lambda_n}}.
\]
It is immediately clear, that in case of independent observations where $\Sigma_n = \sigma^2 \idmat_n$, one has $f\left(0\right) = \sigma^2$, which reproduces \eqref{eq:db}. In the general case, note that
\[
f\left(0\right) = \sum_{h \in \mathbb Z} \gamma\left(h\right),
\]
i.e. the detection boundary solely depends on the value of the spectral density at zero which is known as \textit{long-run variance}.\footnote{The long-run variance of a process $(Z_t)_{t	\in \mathbb Z}$  with spectral density $f$ is defined as $\lim\limits_{n\to\infty} n^{-1}\Var\left[S_n\right]$, where $S_n=\sum\limits_{i=1}^n Z_i$. It holds that $\lim\limits_{n\to\infty} n^{-1}\Var [S_n]=f(0)$ (see \cite{ibragimov_independent_1971} and Section~\ref{app:ibragimov} for details).}

In case of the AR(1)-based noise $\xi_n:=(1-\rho^2)^{1/2} (Z_1,\dots,Z_n)^T$ with unit variance as shown in Figure~\ref{fig:model}, the auto-covariance of the underlying AR(1) process $Z_t$ is given by $\gamma_Z\left(h\right) = \gamma_Z(0)\rho^{\left|h\right|}$, where $\gamma_Z(0)=(1-\rho^2)^{-1}$. Thus the spectral density at zero of the noise process $\xi=\left((1-\rho^2)^{1/2}Z_i\right)_{i\in\N}$ is 
\[
f_\xi\left(0\right) = \left(1-\rho^2\right) \sum_{h=-\infty}^{\infty} \gamma_Z(h) = \frac{1+\rho}{1-\rho},
\]
and hence the detection boundary is given by
\begin{equation}\label{eq:AR1_db}
\Delta_n\asymp \sqrt{\frac{1+\rho}{1-\rho}} \sqrt{\frac{-2\log\lambda_n}{n \lambda_n}}.
\end{equation}
As an immediate consequence, this shows that bump detection is easier under a negative correlation $\rho$ than in case of positive correlations. For the three values employed in Figure \ref{fig:model} we compute for the factor $\sqrt{\frac{1+\rho}{1-\rho}}$ in \eqref{eq:AR1_db} the values $2.38$ when $\rho = 0.7$ and $0.42$ when $\rho = -0.7$. This means that the amplitude of detectable signals for $\rho=0.7$ and $\rho=-0.7$ differs approximately by a factor of $5.6$. Also, given the bump length $\lambda_n$, the detection of a bump of the same size $\Delta_n$ for $\rho = 0.7$ requires approximately a $6$ times larger sample size than for $\rho= 0$, and even a $31$ times larger sample size than for $\rho = -0.7$. This is in good agreement with the intuitive findings from Figure \ref{fig:model} and confirmed in finite sample situations in Section \ref{sec:simulations}. In the simulations we also investigate the influence of several bumps instead of one, and find that independent of $\rho$, multiple bumps always help detection, as to be expected.

Remarkably, as in the case of i.i.d. noise with variance $\sigma^2$, where we have $f(0)=\sigma^2$, certain dependent error processes might also satisfy $f(0)=\sigma^2$, and hence obey the same difficulty to detect a bump as for the independent case. As an example, consider the stationary and causal AR$(2)$ process given by $Z_t=\frac{1}{2}Z_{t-1}-\frac{1}{2}Z_{t-2}+\zeta_t$, where $\zeta_t\sim\mathcal{N}(0,1)$ for $t\in\mathbb{Z}$. In this case $f_Z\left(0\right) = \frac12 - \frac12 + 1 = 1$, even though the process $Z_t$ is clearly not independent (see Section \ref{sec:arma_finite} for a comprehensive treatment of ARMA processes).

\paragraph{Proof strategy.} To prove a lower detection bound, we will employ techniques dating back to \citet{i93} and \citet{ds01} developed for independent observations. To generalize this approach to our dependent case, we will use a recent weak law of large numbers due to \citet{wanghu14} for triangular arrays of random variables that are non-independent within each row and non-identically distributed between rows (see Section \ref{app:wlln} for the precise statement and also \citep{gut91,cabrera2005,sung08,shen2017} for related results).

For the upper detection bound, we will provide an explicit test based on the supremum of the moving average process $\left(\1_I^T Y\right)_{I \in \mathcal I \left(\lambda_n\right)}$. A valid critical value will be given based on a chaining technique. Note that this cannot be obtained by a continuous upper bound of the stochastic process (as e.g. provided in Theorem 6.1 in \citep{ds01}) due to the fact that the dependency structure is allowed to change with $n$ and hence there is no continuous analog of $\left(\1_I^T Y\right)_{I \in \mathcal I \left(\lambda_n\right)}$.

\subsection{Related work}\label{subsec:related}

Bump detection for dependent data appears to be relevant to a variety of applications where piece-wise constant signals (i.e. several bumps) are observed under dependent noise. Exemplary, we mention molecular dynamics (MD) simulations, where collective motion characteristics of protein atoms are studied over time (see e.g. \citep{ketal12} and the references therein). For certain proteins it has been shown that the noise process can be well modeled by a stationary ARMA($p,q$) process with small $p$ and $q$, see \citep{skm17}. Another application is the analysis of ion channel recordings, where one aims to identify opening and closing states of physiologically relevant channels (see \citep{ns95} and the references therein). Here, the dependency structure is induced by a known band-pass filter, ensuring that $\Sigma_n$ in \eqref{eq:model_dep} can be precomputed explicitly (which corresponds to our setting of known $\Sigma_n$), and allowing for a good approximation by stationary and $m$-dependent noise with small $m$, see \citep{petal18}. 

In fact, bump detection as discussed here is closely related to estimation of a signal which consists of piece-wise constant segments, often denoted as change point \textsl{estimation}. We refer to the classical works of \citet{ibragimov1981}, \citet{ch97}, \citet{bd93}, \citet{cms94}, and \citet{Siegmund1985} for a survey of the existing results as well as to the review article by \citet{Aue2013}. Indeed, if the bumps have been properly identified by a detection method, posterior estimation of the signal is relatively easy, see \citep{fms14} for such a combined approach in case of i.i.d. errors, and \citep{d18} in case of dependent data. We also mention \citep{chakar2017}, who presented a robust approach for AR(1) errors.

Model \eqref{eq:model_dep} can be seen as prototypical for the more complex situation when several bumps are to be detected. We do not intend to provide novel methodology for this situation in this paper, rather Theorem \ref{thm:mainthm(poly)} provides a benchmark for \textsl{detecting} such a bump which then can be used to benchmark the detection power of any method designed for this task. Minimax detection has a long history, see e.g. the seminal series of papers by \citet{i93} or the monograph by \citet{t09}. More recently, \citet{goldenshluger2015} provided a general approach based on convex optimization. In case of independent observations, the problem of detecting a bump has been considered in \citep{b02,boysen2009,fms14,cw13,dw08,jcl10}, and our strategy of proof for the lower bound is adopted from \cite{ds01}.  We also mention \citep{enikeeva2018} for a model with a simultaneous bump in the variance, and \citep{psm17} for heterogeneous noise, however still restricted to independent observations.

The literature on minimax detection for dependent noise is much less developed, and most similar in spirit to our work are the papers by \citet{hj10} and \citet{Keshavarz2018}. In the former, the minimax detection boundary for an unstructured version of the model \eqref{eq:model_dep} in a Bayesian setting is derived, that is $\Prob{}{m_n\left(\frac{i}{n}\right)  = \Delta_n} = \rho_n$ and $\Prob{}{m_n\left(\frac{i}{n}\right)  = 0} = 1-\rho_n$ with a probability $\rho_n$ tending to $0$. 
%This problem can be seen as a dependent version of the celebrated needle in a haystack problem (see e.g. \citep{achz08,dj04,is02,cw14,cjj11,b02}), where the goal is to identify a small number $\lambda_n$ of non-zero components in the mean of a sparse high-dimensional multivariate normal vector. If additionally the location of the bump is known, this problem is closely related to the classical problem of testing the difference in mean between two  samples of Gaussian observations with the same covariance matrix \citep{Tony_Cai_Liu_Xia_2014}. 
In contrast to \cite{hj10}, in the present setting we can borrow strength from neighboring observations in a bump. Still, we can exploit a result in \citep{hj10} about the decay behavior of inverses of covariance matrices (see Section \ref{app:precision_matrix}) to validate Assumption \ref{cond:P}. \citet{Keshavarz2018} deal with the classical change-point in mean problem, i.e. with the  problem to detect whether $m_n(i/n)\equiv 0$ for all $1 \leq i \leq n$, or if there exists $\tau \in\left[1,n\right]$ such that $m_n(i/n)=-\frac12\Delta_n\1\{i\le\tau\}+\frac12\Delta_n\1\{i>\tau\}$ for $1 \leq i \leq n$. The authors derive upper and lower bounds for detection from dependent data as in \eqref{eq:model_dep}, similar in spirit to our Theorem \ref{thm:mainthm(poly)}. Their bounds, however, do not coincide with ours, i.e. they do not derive the precise minimax detection boundary, as they are mostly interested in the rate of estimation. However, as we see from Theorem \ref{thm:mainthm(poly)}, the $\sqrt{-\log\lambda_n}$ rate does not change, it is the constant $f(0)$ which matters. We will employ several of their computations concerning covariance structures of time series (while correcting a couple of technical inaccuracies).

We finally comment on the assumption of knowing $\Sigma_n$ and the length $\lambda_n$. If $\lambda_n$ is unknown, estimation of the function $m_n$ can be performed in the independent noise case by SMUCE \citep{fms14} via a multiscale approach. SMUCE is known to achieve the asymptotic detection boundary \eqref{eq:db} in case of  i.i.d. Gaussian errors. For the dependent case with a (partially) unknown covariance matrix $\Sigma_n$, further methods for estimation of $m_n$ such as H-SMUCE \citep{psm17}, J-SMURF \citep{hotzetal2013} or JULES \citep{petal18} have been developed. They all rely on a local estimation of the covariance structure in combination with a multiscale approach. None of these methods achieves the detection boundary derived in this paper, and hence it remains unclear if not knowing $\Sigma_n$ and  / or $\lambda_n$ would affect it. Developing a test which achieves a corresponding upper bound by multiscale methods is beyond the scope of this paper and is postponed to future work.

\subsection{Organization of the paper}

The remaining part of this paper is organized as follows: In Section \ref{sec:overview} we give a precise statement of our assumptions and formulate our main theorem. Also non-asymptotic results are discussed here. The implications for ARMA models are then given in Section \ref{sec:arma_finite}, where the previously mentioned non-asymptotic results are specified for AR($p$) noise. In Section \ref{sec:simulations} we present some simulations which support that our asymptotic theory is already useful for small samples. All proofs are deferred to Section \ref{sec:proofs}.

\section{Main results}\label{sec:overview}

\subsection{Notation and assumptions}\label{subsec:notations}

To treat the testing problem \eqref{eq:testing_problem}, we will consider tests $\Phi_n : \R^n\to \left\{0,1\right\}$, $n\in\N$, where $\Phi_n\left(Y\right) = 0$ means that the {null} hypothesis {$H_0$} is accepted, and $\Phi_n\left(Y\right) = 1$ means that the {null} hypothesis is rejected, i.e. the presence of a bump is concluded.

Denote by $\mathbb P_{0}$ the measure $\mathcal N_n\left(0, \Sigma_n\right)$ of $Y$ under the null hypothesis and by  $\mathbb P_{I, \delta}$  the  measure $\mathcal N_n(\delta\1_{I},\Sigma_n)$ of $Y$ given that there is a bump of height $\delta$ within the interval $I$. With this we will denote the corresponding expectations accordingly by $\mathbb E_{0}$ and $\mathbb E_{I, \delta}$. We define the type I error of $\Phi_n$ by
\[
\bar\alpha \left(\Phi_n, \Sigma_n\right) := \E{0}{\Phi_n\left(Y\right)} = \Prob{0}{\Phi_n\left(Y\right) =1}.
\]
Furthermore, we say that a sequence $(\Phi_n)_{n\in\N}$ of such tests has asymptotic level $\alpha \in \left[0,1\right]$ if $\limsup_{n \to \infty} \bar\alpha \left(\Phi_n, \Sigma_n\right)\leq \alpha$. The type II error depending on  the parameters $\Sigma_n,\Delta_n$ and $\lambda_n$ is defined as 
\[
\bar\beta \left(\Phi_n, \Sigma_n,\Delta_n,\lambda_n\right) := \sup_{I\in\mathcal{I}(\lambda_n)} \sup_{|\delta|\ge \Delta_n} \Prob{I,\delta}{\Phi_n\left(Y\right) = 0}.
\]
For a sequence $(\Phi_n)_{n\in\N}$ of such tests we define its asymptotic type II error to be $\limsup_{n \to \infty}\bar\beta \left(\Phi_n, \Sigma_n, \Delta_n,\lambda_n\right)$. The asymptotic power of such a family is then given by $1-\limsup_{n \to \infty}\bar\beta \left(\Phi_n,  \Sigma_n,\Delta_n,\lambda_n\right)$. For the sake of brevity, we might suppress the dependency on the parameters in the following and write only $\bar\alpha \left(\Phi_n\right)$ and $\bar \beta \left(\Phi_n\right)$, respectively.

With this notation, we can now precisely recall the requirements for lower and upper bounds on detectability as discussed in the introduction:

\begin{description}
	\item[upper detection bound:] For any $\alpha\in\left(0,\frac12\right)$, there exist $c^*>0$ and a sequence of tests $\Phi_{n,\alpha}^*$, $n\in\N$ of asymptotic level $\alpha$ such that $\forall c>c^*$, 
	$$
	\limsup_{n\to\infty} \bar\beta \left(\Phi_n, \Sigma_n,c\varphi_n,\lambda_n\right)\leq \alpha.
	$$
	Note that this notion of the upper detection bound is in accordance with the usual minimax testing paradigm (cf. \citet{Ingster&Suslina:2003}), as it implies that
	\[
	\lim_{n\to\infty} \inf_{\Phi\in\Psi_n}[\bar\alpha \left(\Phi, \Sigma_n\right)+\bar\beta \left(\Phi, \Sigma_n,c\varphi_n,\lambda_n\right)]=0,
	\]
	as $n\to\infty$, since $\alpha$ was arbitrary. Here, $\Psi_n$ is the collection of all tests for the testing problem \eqref{eq:testing_problem} given $n$ observations.
	\item[lower detection bound:] For any $\alpha\in\left(0,\frac12\right)$, there exists $c_*>0$ such that $\forall c<{c_*}$,  and for any sequence of tests $\Phi_n$, $n\in\N$ of asymptotic level $\alpha$, 
	$$
	\liminf_{n\to\infty} \bar\beta \left(\Phi_n, \Sigma_n,c\tilde \varphi_n,\lambda_n\right)\geq 1-\alpha.
	$$
	This implies that
	\[
	\lim_{n\to\infty} \inf_{\Phi\in\Psi_n}[\bar\alpha \left(\Phi, \Sigma_n\right)+\bar\beta \left(\Phi, \Sigma_n,c\varphi_n,\lambda_n\right)]=1.
	\]
\end{description}
The choice of $1-\alpha$ as the lower bound of the limit of the type II errors in the lower detection bound is justified by the fact that the minimax testing risk is bounded from below as follows (see \cite{Ingster&Suslina:2003}, p.~55, Theorem~2.1):
$$
\inf_{\Phi\in\Psi} [\bar\alpha \left(\Phi, \Sigma_n\right)+\bar\beta \left(\Phi, \Sigma_n,c\varphi_n,\lambda_n\right)] \ge 1- \frac12 \|[\mathcal P_0],[\mathcal P_1]\|_1,
$$
where $ \|[\mathcal P_0],[\mathcal P_1]\|_1$ is the $L_1$-distance between the convex hulls of measures corresponding to the null and the alternative hypotheses and $\Psi$ is the set of all possible tests. It implies that the type II error of the $\alpha$-level test will be always greater or equal $1-\alpha$ for non-distinguishable null and alternative hypotheses.

%To ease the following notation, we will use some asymptotic relations. For two sequences $\left(a_n\right)_{n \in \N}$ and $\left(b_n\right)_{n \in \N}$ we write $a_n \precsim b_n$ if there exists an $N \in \N$ such that $a_n \leq b_n$ for all $n \ge N$. This notation means that $a_n$ is asymptotically less or equal $b_n$. Similarly we define $a_n \succsim b_n$. As stated before, if $\lim_{n\to\infty} a_n/b_n=1$, we write $a_n\asymp b_n$.
%
%

To derive lower and upper bounds in this sense, we will now pose some assumptions on the possible lengths $\lambda_n$ of intervals and the covariance structure $\Sigma_n$:
\begin{Assumption}\label{cond:I}
	We assume that
	\begin{enumerate}
		\item[(i)] {$\frac{n\lambda_n}{\log n}\to\infty$ as $n\to\infty$,}
		\item[(ii)] {$\lambda_n=o\left(\frac{1}{\log n}\right)$ as $n\to\infty$.}
	\end{enumerate}
\end{Assumption}

The first part of Assumption \ref{cond:I} assures that the number of observations within any interval of length $\lambda_n$ {is at least of logarithmic order} as $n\to\infty$. The second condition of Assumption \ref{cond:I}, however, gives a bound for the maximal length of the considered intervals{, which ensures less than $n/\log n$ observations in the bump interval. Roughly speaking both conditions are required to have enough complementary observations (outside respectively inside the bump) to guarantee asymptotic detection.}  Note that, in particular, {Assumption \ref{cond:I}(ii)} means that $\lambda_n\to 0$ as $n\to\infty$, i.e. Assumption \ref{cond:I} especially implies \eqref{eq:I}. We emphasize that conditions as in (ii) restricting $\lambda_n$ from being too large are common. {Assumption~\ref{cond:I} plays a crucial role in the proof of the upper bound, whereas the lower bound can be established under milder conditions~\eqref{eq:I}.}
%Therein it is shown that in the case of vanishing signals (i.e. $\lambda_n\to 0$, which is, of course, the case that we consider in this paper) the SMUCE attains the optimal upper detection bound, when $\sup_n q_n(\varepsilon_n\sqrt{-\log \lambda_n})<1$, where $(\varepsilon_n)_n$ is a sequence of real numbers such that $\varepsilon_n\sqrt{-\log\lambda_n}\to \infty$, and $q_n$ may be, for example, chosen to be the $\alpha$-quantile of the limiting distribution of SMUCE.

However, note that when we consider a slightly modified version of the testing problem \eqref{eq:testing_problem} where the bump may not occur in any interval of length $\lambda_n$, but only within a candidate set $I_k := [(k-1)\lambda_n, k\lambda_n)$, $1\leq k\leq \lfloor 1/\lambda_n\rfloor$ of non-overlapping intervals, then Assumption \ref{cond:I} can be replaced by \eqref{eq:I} and the detection boundary will remain the same (cf. Section \ref{sec:AR_non_asympt}). 

Instead of posing assumptions on $\Sigma_n$ directly, we will again employ the spectral density $f$ of the underlying stationary process $Z$ as mentioned in the introduction. To do so, we require some more terminology. For a function $g\in\mathbf L_2\left[-1/2,1/2\right)$, we denote by $\mathcal{T}(g)$ the Toeplitz matrix {generated by } $g$, i.e. the matrix with entries $ \left(\mathcal{T}(g)\right)_{i,j \in \N}= g_{j-i}$, where
\[
g_k=\int_{-\frac{1}{2}}^{\frac{1}{2}} g(u)e^{-2\pi i k u}du,  \qquad k\in\mathbb{Z},
\]
is the $k$-th Fourier coefficient of $g$. Note that this allows us to encode the covariance matrix $\Sigma_n$ completely in terms of $f$. More precisely, the covariance matrix $\Sigma_n$ of the noise $\xi_n$ in \eqref{eq:model_dep}
%if the noise $\xi_n$  in \eqref{eq:model_dep} is a sample of $n$ consecutive realizations of the process $Z_t$, then its covariance matrix 
has entries $\Sigma_n (i,j)=\gamma(|i-j|) = f_{\left|i-j\right|}$, and we see that $\Sigma_n=:\mathcal{T}_n(f)$ is the $n$-th truncated Toeplitz matrix {generated by} $f$, i.e. the upper left $n\times n$ submatrix, of $\mathcal{T}(f)$. Consequently, we will also pose the corresponding assumptions in terms of the function $f$, which allows us to derive results for any sequence $\left(\Sigma_n\right)_{n\geq 1}$ of covariance matrices which are generated by such an $f$ (and not only for specific dependent processes):
\begin{Assumption}\label{cond:P}
	Let $(\Sigma_n)_{n\geq 1}$ be a sequence of covariance matrices such that $\Sigma_n= \mathcal{T}_n(f)$ as introduced above with a function $f:[-1/2, 1/2) \to\R$, that is continuous  and satisfies $\lim_{\nu\to 1/2}f(\nu)=f(-1/2)$ and $\essinf_{\nu\in [-1/2,1/2)} f(\nu)>0$.
	Further, suppose that the Fourier coefficients $f_h$, $h\in\mathbb{Z}$ of $f$ decay sufficiently fast, i.e. there are constants $C>0$ and $\kappa>0$, such that
	\[
	|f_h|\leq C(1+|h|)^{-(1+\kappa)}, \quad h\in\mathbb Z.
	\]
\end{Assumption}

Assumption \ref{cond:P} ensures that the dependency between $\1_I^TY$ and $\1_{I'}^TY$ for two candidate intervals $I, I' \in \mathcal I\left(\lambda_n\right)$, will {{} be} asymptotically small as soon as they are disjoint. It excludes trivial cases such as total dependence described by $\Sigma_n\left(i,j\right) = 1$ for all $i,j \in \left\{1,...,n\right\}$, but also permits spectral densities $f$ with only slowly decaying Fourier coefficients such as discontinuous functions.

Note that also sequences of covariance matrices of the form $(\Sigma_n)_{i,j}=g\left(\frac{|i-j|}{n}\right)$, $1\leq i,j\leq n$, $n\in\N$, where $g$ is some kernel function, are prohibited due to this assumption. Covariance matrices of this kind would have the undesired effect to make the dependency between  $\1_I^TY$ and $\1_{I'}^TY$ even for disjoint candidate intervals $I,I' \in \mathcal I \left(\lambda_n\right)$  stronger as the length $\lambda_n$ vanishes.

\subsection{Asymptotic detection boundary}\label{sec:result}

Our main theorem will be the following.

\begin{Theorem}\label{thm:mainthm(poly)}
	If Assumptions \ref{cond:I} and \ref{cond:P} hold for the bump regression model \eqref{eq:model_dep}, then the asymptotic minimax detection boundary for the testing problem \eqref{eq:testing_problem} is given by
	\[
	\Delta_n\asymp \sqrt{\frac{-2 f(0)\log\lambda_n}{n\lambda_n}},
	\]
	as $n\to \infty$.
\end{Theorem}

For the details of the proof we refer to Section \ref{sec:proofs}. The upper bound will be achieved by a specific test $\Phi_n^{\mathrm{a}}$, which scans over all intervals of length $\lambda_n$, given by
\begin{equation}\label{eq:Phi_full}
\Phi_n^{\mathrm{a}}(Y)= \begin{cases} 1 & \text{if } \sup_{I\in\mathcal{I}(\lambda_n)}\frac{\left|\1_I^TY\right|}{\sqrt{\1_I^T\Sigma_n\1_I}}>c_{\alpha,n}, \\ 0 & \text{else}, \end{cases}
\end{equation}
where the threshold $c_{\alpha,n}$ will be determined in the proof of Lemma \ref{lm:app_upper_bound} in Section \ref{sec:proofs}. Note that this test is not a likelihood ratio type test (as the LRT relies on $\1_I^T\Sigma_n^{-1}Y$ instead of $\1_I^TY$). 

For the proof of the lower bound we employ a strategy from \citet{ds01}, and use a very specific law of large numbers for arrays of non-independent and non-identically distributed random variables.

\subsection{Non-asymptotic results}\label{sec:AR_non_asympt}

Note that Theorem \ref{thm:mainthm(poly)} yields only an asymptotic result. In this section we give non-asymptotic results in the case of a seemingly simpler testing problem with possible bumps that belong to a set of non-overlapping intervals. This is formalized by considering the set $\mathcal I^0$ of non-overlapping candidate intervals given by
\begin{equation}\label{eq:I0}
\mathcal I^0:= \left\{I_k~\big|~ 1 \leq k \leq \lfloor \lambda_n^{-1} \rfloor \right\}, \qquad I_k:=\bigl[(k-1)\lambda_n, k\lambda_n\bigr), \quad 1 \leq k \leq \lfloor \lambda_n^{-1} \rfloor.
\end{equation}
The goal is still to detect the presence of the bump (but with position being only in $\mathcal I^0$) and to derive non-asymptotic results on the detection boundary for the testing problem
\begin{gather}
H_0: Y \sim \mathcal N_n\left(0,\Sigma_n\right)\nonumber \\
\mbox{against}  \label{eq:testing_problem_non_overlap} \\ 
H_1^n: \exists 1 \leq k \leq \lfloor \lambda_n^{-1} \rfloor ,\ \exists\delta\in\R: |\delta|\ge \Delta_n\quad \mbox{such that} \quad Y \sim \mathcal N_n\left(\delta\1_{I_k},\Sigma_n\right)\nonumber
\end{gather}
Note that this testing problem might seem much simpler than \eqref{eq:testing_problem} at a first glance, but we will see, however, that the (asymptotic) detection boundary is in fact the same. Concerning {{}the} lower bound, this can be seen readily from the proof of Theorem \ref{thm:mainthm(poly)}, cf. Lemma \ref{lm:app_lower_bound}.

To detect a bump, we will here employ the maximum likelihood ratio test
\[
\Phi_n^{\mathrm{d}}(Y)=\1\left\lbrace T_n^0(Y)>c_{\alpha,n}\right\rbrace
\]
based on the statistic
\begin{equation}\label{eq:LRT_nonover}
T_n^0(Y)=\sup_{ I\in \mathcal I^0 } \frac{|\1_I^T\Sigma_n^{-1}Y|}{\sqrt{\1_I^T\Sigma_n^{-1}\1_I}} = \sup_{1\le k\le \lfloor \lambda_n^{-1}\rfloor}  \frac{|\1_{I_k}^T\Sigma_n^{-1}Y|}{\sqrt{\tilde\sigma_k}},
\end{equation}
where we denote
\begin{equation}\label{eq:sigmak}
\tilde\sigma_k=\1_{I_k}^T\Sigma_n^{-1}\1_{I_k},\quad k=1,\dots,\lfloor \lambda_n^{-1}\rfloor.
\end{equation}
The quantities $\tilde\sigma_k$ are in fact the variances of $\1_I^T\Sigma_n^{-1}Y$ corresponding to the sum of $\lfloor n \lambda_n\rfloor$ random variables with covariance structure given by the $I_k$-block of $\Sigma_n^{-1}$. The type I and II errors of the test $\Phi_n^{\mathrm{d}}$ are defined as 
\[
\tilde\alpha \left(\Phi_n^{\mathrm{d}}\right) :=  \Prob{0}{\Phi_n^{\mathrm{d}}\left(Y\right) =1} \quad\mbox{and}
\quad \tilde \beta \left(\Phi_n^{\mathrm{d}}\right) := \sup_{I\in\mathcal{I}^0} \sup_{|\delta|\ge \Delta_n} \Prob{I,\delta}{\Phi_n^{\mathrm{d}}\left(Y\right) = 0}.
\]

Then the following result establishes basic properties of the test $\Phi_n^{\mathrm{d}}$.
\begin{Theorem}\label{thm:general:nonoverlap}
	Consider the testing problem~(\ref{eq:testing_problem_non_overlap}) and let $\alpha\in(0,1)$ be any fixed significance level. For the maximum likelihood ratio test $\Phi_n^{\mathrm{d}}$ set
	\begin{equation}\label{eq:test_thresh}
	c_{\alpha,n}:=\sqrt{2 \log \frac{2}{\alpha \lambda_n}}.
	\end{equation}
	Then it holds $\tilde\alpha \left(\Phi_n^{\mathrm{d}}\right)\leq \alpha$ for all $n \in \mathbb N$ and
	\[\tilde\beta \left(\Phi_n^{\mathrm{d}}\right) \leq \Prob{}{|Z|>\Delta_n\inf_{1\leq k\leq \lfloor\lambda_n^{-1}\rfloor}\sqrt{\tilde\sigma_k}-c_{\alpha,n}}
	\]
	with a standard Gaussian random variable $Z \sim \mathcal N \left(0, 1\right)$ and $\tilde\sigma_k$ as in \eqref{eq:sigmak}.
\end{Theorem}
The proof is obtained by straightforward computations, see Section \ref{sec:proofs} for details. Theorem \ref{thm:general:nonoverlap} yields explicit non-asymptotic bounds for the test $\Phi_n^{\mathrm{d}}$, but those do also yield an asymptotic upper bound for the detection boundary:
\begin{Corollary}\label{cor:disjoint}
	Let $\left(\varepsilon_n\right)_{n\in\N}$ be a positive sequence satisfying
	\begin{equation}\label{eq:epsn}
	\varepsilon_n \sqrt{-\log\lambda_n} \ge \sqrt{\log\frac2\alpha}+\sqrt{\log\frac 1\alpha},
	\end{equation}
	and suppose that the bump altitude $\Delta_n$ in the testing problem \eqref{eq:testing_problem_non_overlap} obeys
	\begin{equation}\label{eq:upper_det_bound}
	\Delta_n\inf_{1\leq k\leq \lfloor \lambda_n^{-1} \rfloor}\sqrt{\tilde\sigma_k}  \geq\sqrt 2\left(1+\varepsilon_n\right) \sqrt{-\log \lambda_n}.
	\end{equation}
	Then the asymptotic type II error of $\Phi_n^{\mathrm{d}}$ with $c_{\alpha,n}$ as in \eqref{eq:test_thresh} satisfies
	\[
	\limsup_{n \to \infty} \tilde \beta \left(\Phi_n^{\mathrm{d}}\right) \leq \alpha,
	\]
\end{Corollary}
This shows that the upper bound to be obtained by $\Phi_n^{\mathrm{d}}$ depends only on the asymptotic behavior of $\inf_{1\leq k\leq \lfloor \lambda_n^{-1} \rfloor}\sqrt{\tilde\sigma_k}$ with $\tilde\sigma_k$ as in \eqref{eq:sigmak}. Inspecting the proof of Lemma \ref{lm:app_lower_bound}, we find that we can derive an according lower bound depending only on the asymptotic behavior of $\sup_{1 \leq k \leq \lfloor \frac{1}{\lambda_n}\rfloor} \sqrt{\tilde\sigma_k}$. In case of AR($p$) noise we will see in Section \ref{sec:AR_non_asympt} that these quantities can be computed explicitly and will asymptotically equal in agreement with Theorem \ref{thm:mainthm(poly)}.

\section{ARMA processes and finite sample results}\label{sec:arma_finite}
\subsection{Application to ARMA processes}
Suppose that the noise vector $\xi_n=(Z_1,\dots,Z_n)^T$ in  (\ref{eq:model_dep}) is sampled from $n$ consecutive realizations of a stationary ARMA$(p,q)$ time series $Z_t$, with $p\geq 0$, $q\geq 0$ defined as
\begin{equation}\label{eq:arma_polynomials}
\phi(B)Z_t=\theta(B)\zeta_t, \quad \zeta_t\stackrel{\text{i.i.d.}}{\sim}\mathcal{N}(0,1), \quad t\in\mathbb{Z}.
\end{equation}
Here  $B$ is the so-called backshift operator, defined by $BX_t=X_{t-1}$, and $\phi(z)$ and $\theta(z)$, $z\in\mathbb{C}$, are polynomials of degrees $p$ and $q$, respectively, given by
\begin{equation}\label{eq:poly}
\phi(z)= 1+\sum_{i=1}^p \phi_iz^i, \quad \theta(z)= 1+\sum_{i=1}^q \theta_iz^i.
\end{equation}
We further suppose that $\phi$ and $\theta$ have no common roots, and that all roots of both $\phi$ and $\theta$ lie outside of the unit circle $\lbrace z\in\mathbb{C} : |z|\leq 1\rbrace$ (see \cite{Brockwell1991} for more details). %Then the corresponding ARMA process is reversible, satisfies Assumption \ref{cond:P} and allows for a representation as an MA$(\infty)$ process.

Denote by $\gamma$ the auto-covariance function of $Z$, i.e. $\gamma(h)=\E{}{Z_tZ_{t+h}}$ for $h\in\mathbb{Z}$ (as clearly $\E{}{Z_t}=0$ for all $t\in\mathbb{Z}$). It is well-known (see for example \cite{Brockwell1991}, Theorem 4.4.2),
that in the case of an ARMA$(p,q)$ time series, its spectral density is given by
\begin{equation}\label{eq:spectral_density}
f(\nu)=\frac{|\theta(e^{-2\pi i\nu})|^2}{|\phi(e^{-2\pi i\nu})|^2},\quad \nu\in [-1/2,1/2).
\end{equation}

Note that the spectral density  $f$  is continuous at 0 as well as the function $1/f$, since the process is reversible and causal under the posed assumptions on $\phi$ and $\theta$. Thus, applying Theorem \ref{thm:mainthm(poly)} to this setting immediately yields the following:
\begin{Theorem}\label{thm:arma_result}
	Assume that we are given observations from \eqref{eq:model_dep}, where the noise $\xi_{n}$ is given by $n$ consecutive samples of an ARMA$(p,q)$ time series as in \eqref{eq:arma_polynomials} with the polynomials $\phi$ and $\theta$ in \eqref{eq:poly} having no common roots and no roots within the unit circle. Furthermore, assume that Assumption \ref{cond:I} holds. Then the asymptotic detection boundary of the hypothesis testing problem \eqref{eq:testing_problem} is given by
	\begin{equation}\label{eq:db_dependent}
	\Delta_n\asymp \sqrt{\frac{-2{f(0)}\log\lambda_n}{n\lambda_n}} = \left|\frac{1+\sum_{i=1}^q \theta_i}{1+\sum_{i=1}^p \phi_i}\right| \sqrt{\frac{-2\log\lambda_n}{n\lambda_n}},
	\end{equation}
	as $n\to\infty$.
\end{Theorem}
We find that the presence of dependency either eases or loads the bump detection, depending on $f\left(0\right)=\left|\theta(1)/\phi(1)\right|{^2}$ (which is $1$ in the independent noise case). If $f\left(0\right) < 1$, then the detection becomes simpler (and smaller bumps are still consistently detectable), but if $f\left(0\right) > 1$ detection becomes more difficult. For AR(1) noise, this issue was already discussed in the introduction. 

\subsection{Non-asymptotic results for AR($p$)}

In this Section we will derive non-asymptotic results for the specific case of AR($p$) noise. Let us therefore specify \eqref{eq:arma_polynomials} to a stationary AR($p$) process $Z_t$,
\begin{equation}\label{eq:ARp}
\sum_{i=0}^p \phi_i Z_{t-i}=\zeta_t,\quad  t\in \mathbb Z
\end{equation}
with independent standard Gaussian innovations $\zeta_t$. In the notation of \eqref{eq:arma_polynomials}, we have $\phi(z)= \sum_{i=0}^p \phi_iz^i$ and $\theta(z)\equiv 1$. Again, we work under the standard assumptions that the characteristic polynomial  $\phi(z)$ has no zeros inside the unit circle $\{z\in \mathbb C:\ |z|\le 1\}$. Note that in this case $f(0)=\left|\sum_{i=0}^p \phi_i\right|^{-2}$.

We have seen in the discussion succeeding Theorem \ref{thm:general:nonoverlap} that the upper and lower bounds depend on the quantities $\tilde\sigma_k=\1_{I_k}^T\Sigma_n^{-1}\1_{I_k}$ and correspondingly, their minimal and maximal values. Theorem~\ref{thm:arma_result} gives the detection boundary condition for ARMA noise with an asymptotic risk constant. Since $\tilde\sigma_k$ is just the sum over the block of $\Sigma_n^{-1}$, using the exact inverse of $\Sigma_n$ (see the appendix for the exact formula of $\Sigma_n^{-1}$ obtained by \citet{siddiqui1958}), we can calculate the minimax risk constants exactly.

\begin{Lemma}\label{lm:finite_beta}
	Let $\Sigma_n$ be the auto-covariance matrix induced by an AR($p$) process $Z_t$ and $\tilde\sigma_k=\1_{I_{k}}^T \Sigma_n^{-1}\1_{I_{k}}$, $k=1,\dots,\lfloor \lambda_n^{-1}\rfloor$. Assume that $1\le \lfloor n\lambda_n\rfloor\le n-2p$ and $n> 3p$. 
	\begin{enumerate}
		\item  If $\lfloor n\lambda_n\rfloor\le p$, then 
		\begin{align}
		\inf_{1\le k\le \lfloor \lambda_n^{-1} \rfloor}\tilde\sigma_k&=\sum_{i=1}^{\lfloor n\lambda_n\rfloor} \left( \sum_{t=0}^{i-1}\phi_t\right)^2,\label{Minconst1}\\
		\sup_{1\le k\le \lfloor \lambda_n^{-1} \rfloor}\tilde\sigma_k &=  
		\inf_{1\le k\le \lfloor \lambda_n^{-1} \rfloor}\tilde\sigma_k+ \sum_{i=0}^{p-\lfloor n\lambda_n\rfloor} \left( \sum_{t=1}^{\lfloor n\lambda_n\rfloor}\phi_{t+i}\right)^2 + \sum_{i=p-\lfloor n\lambda_n\rfloor}^p \left( \sum_{t=0}^{p-i}\phi_{t+i}\right)^2. \label{Maxconst1}	
		%\Msup\left(\left|I_n\right|\right) &= \lambda_n \sum_{t=0}^p \phi_t^2+2 \sum_{k=1}^{\lambda_n-1} (\lambda_n-k)\sum_{t=0}^{p-k} \phi_t \phi_{t+k},		
		\end{align}
		\item If $p< \lfloor n\lambda_n\rfloor\le n-2p$,  then 
		\begin{align}		
		\inf_{1\le k\le \lfloor \lambda_n^{-1} \rfloor}\tilde\sigma_k&=(\lfloor n\lambda_n\rfloor -p)\left(\sum\limits_{t=0}^{p} \phi_{t} \right)^2 + \sum_{i=1}^p \left(\sum_{t=0}^{i-1} \phi_t \right)^2, \label{Minconst2}\\
		\sup_{1\le k\le \lfloor \lambda_n^{-1} \rfloor}\tilde\sigma_k
		%&= \lambda_n \Bigl(\sum_{t=0}^p \phi_t\Bigr)^2 -2 \sum_{k=1}^{p} k\sum_{t=0}^{p-k} \phi_t \phi_{t+k},\label{Maxconst2}\\
		&=  \inf_{1\le k\le \lfloor \lambda_n^{-1} \rfloor}\tilde\sigma_k+\sum_{i=1}^p\left(\sum_{t=0}^{p-i}\phi_{t+i}\right)^2.\label{Maxconst2}
		%\Minf\left(\left|I_n\right|\right)&=(\lambda_n-p) \left(\sum\limits_{t=0}^{p} \phi_{t} \right)^2+\sum_{k=1}^{p} \left(\sum_{t=0}^{p-k}\phi_t\right)^2.\label{Minconst2}
		\end{align}
	\end{enumerate}
	
\end{Lemma}

We can now use the results of Theorem \ref{thm:general:nonoverlap} and get the exact detection boundaries for two different regimes, when $\lfloor n\lambda_n\rfloor\le p$ and $p<\lfloor n\lambda_n\rfloor\le n-2p$. Note that condition~\eqref{eq:cov_cond1} is automatically satisfied since the inverse covariance matrix $\Sigma_n^{-1}$ is $2p+1$-diagonal.

\begin{Corollary}
	Assume that possible locations $k$ of the bump $I_k\in\mathcal I^0$ are separated from the endpoints of the interval: $p<k<n-p-\lfloor n\lambda_n\rfloor $. Then the upper and lower bound constants match in both cases  and are given by formulas~{(\ref{Minconst1}) and~(\ref{Minconst2})} for the case of $\lfloor n\lambda_n\rfloor\le p$ and $p<\lfloor n\lambda_n\rfloor \le n-2p$, respectively. This follows immediately from the discussion in Section~\ref{subsec:ARp_precision_matrix}, in particular equations \eqref{SBmr} and~\eqref{SBmr_bigr}.
\end{Corollary}

\begin{Remark}
	It seems reasonable, that, in case of bumps of length smaller than $p$, we would need to analyze the type I error with some finer technique than just the union bound. 
	
	On the other hand , we observe that if $n\lambda_n\to \infty$ and $\lambda_n\to 0$ as $n\to\infty$, then 
	$$
	\sup_{1\le k\le \lfloor \lambda_n^{-1} \rfloor}\tilde\sigma_k \asymp n\lambda_n \left(\sum\limits_{t=0}^{p} \phi_{t} \right)^2 
	%= n\lambda_n \phi^2(1) 
	\asymp \inf_{1\le k\le \lfloor \lambda_n^{-1} \rfloor} \tilde\sigma_k,
	$$
	in accordance with Theorem \ref{thm:arma_result}.
\end{Remark}

\section{Simulations}\label{sec:simulations}

In this Section we will perform numerical studies to examine the finite sample accuracy of the asymptotic upper bounds for the detection boundary. We focus on the situation that the noise $\xi_n$ in \eqref{eq:model_dep} is generated by an AR(1) process, given by $\phi(z)=1-\rho z$ and $\theta(z)\equiv 1$ (in the notation of \eqref{eq:arma_polynomials}), where $|\rho|<1$. More precisely, the AR(1) process is given by the equation $Z_t-\rho Z_{t-1}=\zeta_t$, $t\in\Z$ where $\zeta_t\sim \mathcal N(0,1)$ are i.i.d.. Note the slight difference to the setting considered in the introduction and Figure \ref{fig:model}, as here the noise does not have standardized margins.

From Theorem \ref{thm:arma_result} we obtain the detection boundary
\begin{equation}\label{eq:boundaryAR1}
\Delta_n \asymp\frac{\sqrt 2}{1-\rho} \sqrt{\frac{-\log \lambda_n}{n \lambda_n}}.
\end{equation}
In the following we fix the value of the detection rate $\sqrt{-\log \lambda_n /(n\lambda_n)}$ in \eqref{eq:boundaryAR1}  to be roughly 1/6 and consider three different situations, namely small sample size ($\lambda_n = 0.1$, $n = 829$), medium sample size ($\lambda_n = 0.05$, $n = 2157$) and large sample size ($\lambda_n = 0.025$, $n = 5312$). Thus, the remaining free parameters are $\rho$ and $\Delta_n$, and the detection boundary \eqref{eq:boundaryAR1} connects them by the asymptotic relation
\begin{equation}\label{eq:rhoDelta}
\Delta_n \asymp \frac{\sqrt2}{1-\rho}\cdot \frac{1}6 \approx \frac{0.236}{1-\rho}.
\end{equation}

Let us now specify the investigated tests. For the (general) test from Section \ref{sec:result}, the critical value $c_{\alpha,n}$ is only given implicitly, cf. \eqref{eq:c_alpha_n}. To simplify, in view of $c_{\alpha,n}=\sqrt{2\log\frac{2}{\alpha\lambda_n}}(1+o(1))$, we will therefore use
\begin{equation}\label{eq:Phia}
\Phi_n^{\mathrm{a}} \left(Y\right):= \begin{cases} 1 & \text{if }\sup\limits_{I \in \mathcal I \left(\lambda_n\right)} \frac{\left|\1_I^TY\right|}{\sqrt{\1_I^T \Sigma_n \1_I}} > \sqrt{2\log\frac{2}{\alpha\lambda_n}}, \\
0 & \text{else} \end{cases}
\end{equation}
as an asymptotic version. Further we would like to investigate the maximum likelihood ratio test relying only on non-overlapping intervals $I_k=\left[\left(k-1\right)\lfloor n \lambda_n\rfloor+1, k \lfloor n \lambda_n\rfloor \right)$ from Section \ref{sec:AR_non_asympt} given by
\begin{equation}\label{eq:Phid}
\Phi_n^{\mathrm{d}} \left(Y\right):= \begin{cases} 1 & \text{if }\sup\limits_{1 \leq k \leq \lfloor \frac{1}{\lambda_n}\rfloor} \frac{\left|\1_{I_k}^T\Sigma_n^{-1}Y\right|}{\sqrt{\1_{I_k}^T \Sigma_n^{-1} \1_{I_k}}} > \sqrt{2\log\frac{2}{\alpha\lambda_n}}, \\
0 & \text{else.} \end{cases}
\end{equation}
Note that the latter requires to scan only over $\lfloor 1/\lambda_n\rfloor $ intervals, whereas the former requires to scan over $ \lfloor n \left(1-\lambda_n\right) \rfloor$ intervals. Consequently, the maximum likelihood ratio test from Section \ref{sec:AR_non_asympt} can be computed faster by a factor of
\[
\frac{n \left(1-\lambda_n\right)}{1/\lambda_n} = n \lambda_n \left(1-\lambda_n\right)
\]
independent of $\Sigma_n$. For the three situations mentioned above this yields values of $\approx 74$ in the small sample regime, $\approx 102$ in the medium sample regime, and $\approx 129$ in the large sample regime. However, our results from Theorem \ref{thm:mainthm(poly)} and the discussion succeeding Theorem \ref{thm:general:nonoverlap} imply, that the testing problems \eqref{eq:testing_problem} and \eqref{eq:testing_problem_non_overlap} are of the same difficulty in the sense that they both have the same separation rate.

In the following we examine the type I and type II errors $\bar \alpha \left(\Phi_{*}\right)$ and $\bar\beta\left(\Phi_{*}\right)$ with $* \in \left\{\mathrm{a},\mathrm{d}\right\}$ by $2000$ simulation runs for $\alpha = 0.05$ with different choices of $\rho$, $n$, $\lambda_n$ and $\Delta_n$. The position $I \in \mathcal I \left(\lambda_n\right)$ is always drawn uniformly at random. Furthermore, we investigate the situation of $2$ and $5$ disjoint bumps within $\left[0,1\right]$. 

The finite sample type I error {of both} $\Phi_n^{\mathrm{a}}$ and $\Phi_n^{\mathrm{d}}$ in all three sample size situations are shown in Figure \ref{fig:level} versus the correlation parameter $\rho\in \left\{-0.99,-0.98,...,0.99\right\}$. We find that $\Phi_n^{\mathrm{d}}$ is somewhat conservative, which is clearly due to the usage of the union bound in deriving the critical value in \eqref{eq:Phid}, see the proof of Theorem \ref{thm:general:nonoverlap}. Opposed, $\Phi_n^{\mathrm{a}}$ is conservative for $\rho > 0$, and liberal for $\rho < 0$. This is clearly due to the simplified critical value in \eqref{eq:Phia}, which is only asymptotically valid, and furthermore the employed asymptotics depend on $\rho$ due to the result by \citet{ibragimov_independent_1971}, see also Section \ref{app:ibragimov}. However, it seems that already in the small sample size regime our asymptotic results provide a very good approximation for both tests.

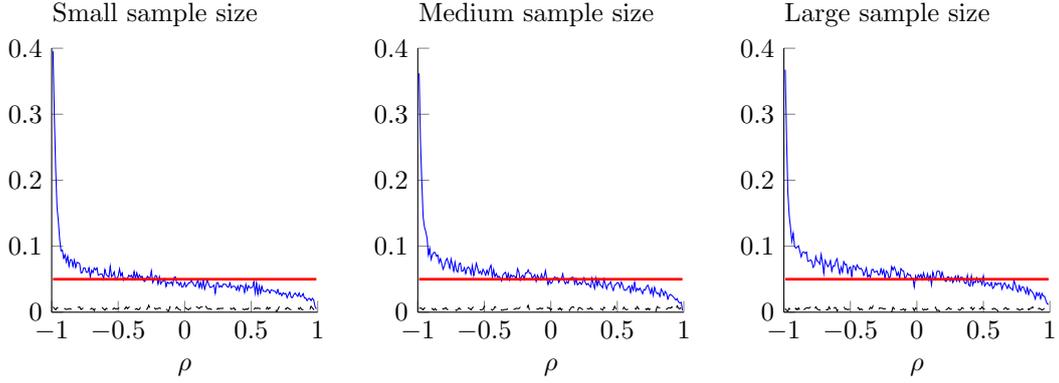
\begin{figure}[!htb]
	\setlength\fheight{3.5cm} \setlength\fwidth{3.5cm}
	\centering
	\begin{tabular}{lll}
		\qquad Small sample size
		&
		\qquad Medium sample size
		&
		\qquad Large sample size
		\\
		\begin{tikzpicture}[baseline]

\begin{axis}[%
width=\fwidth,
height=\fheight,
scale only axis,
xmin=-1,
xmax=1,
xlabel={$\rho$},
ymin=0,
ymax=0.4,
axis x line*=bottom,
axis y line*=left
]
\addplot [color=black, dashed]
  table[row sep=crcr]{%
-0.99	0.008\\
-0.98	0.006\\
-0.97	0.0035\\
-0.96	0.007\\
-0.95	0.0055\\
-0.94	0.0055\\
-0.93	0.008\\
-0.92	0.005\\
-0.91	0.0055\\
-0.9	0.0055\\
-0.89	0.004\\
-0.88	0.0045\\
-0.87	0.0055\\
-0.86	0.005\\
-0.85	0.0045\\
-0.84	0.0065\\
-0.83	0.0055\\
-0.82	0.004\\
-0.81	0.0065\\
-0.8	0.0075\\
-0.79	0.008\\
-0.78	0.004\\
-0.77	0.0055\\
-0.76	0.0065\\
-0.75	0.006\\
-0.74	0.0045\\
-0.73	0.006\\
-0.72	0.0055\\
-0.71	0.004\\
-0.7	0.006\\
-0.69	0.005\\
-0.68	0.0025\\
-0.67	0.0075\\
-0.66	0.0065\\
-0.65	0.006\\
-0.64	0.005\\
-0.63	0.0055\\
-0.62	0.007\\
-0.61	0.005\\
-0.6	0.003\\
-0.59	0.006\\
-0.58	0.005\\
-0.57	0.006\\
-0.56	0.004\\
-0.55	0.0075\\
-0.54	0.0045\\
-0.53	0.0065\\
-0.52	0.005\\
-0.51	0.005\\
-0.5	0.0055\\
-0.49	0.007\\
-0.48	0.0065\\
-0.47	0.0035\\
-0.46	0.007\\
-0.45	0.005\\
-0.44	0.008\\
-0.43	0.0055\\
-0.42	0.0075\\
-0.41	0.0085\\
-0.4	0.006\\
-0.39	0.007\\
-0.38	0.003\\
-0.37	0.0045\\
-0.36	0.0055\\
-0.35	0.0035\\
-0.34	0.0045\\
-0.33	0.0085\\
-0.32	0.007\\
-0.31	0.004\\
-0.3	0.005\\
-0.29	0.0035\\
-0.28	0.0055\\
-0.27	0.0105\\
-0.26	0.005\\
-0.25	0.0055\\
-0.24	0.0045\\
-0.23	0.0085\\
-0.22	0.0045\\
-0.21	0.0055\\
-0.2	0.0055\\
-0.19	0.0065\\
-0.18	0.0055\\
-0.17	0.0045\\
-0.16	0.006\\
-0.15	0.004\\
-0.14	0.0065\\
-0.13	0.0025\\
-0.12	0.002\\
-0.11	0.008\\
-0.1	0.005\\
-0.09	0.0065\\
-0.08	0.003\\
-0.07	0.005\\
-0.0599999999999999	0.008\\
-0.0499999999999999	0.0055\\
-0.0399999999999999	0.004\\
-0.03	0.003\\
-0.02	0.0065\\
-0.01	0.0035\\
0	0.0075\\
0.01	0.006\\
0.02	0.0045\\
0.03	0.0035\\
0.0399999999999999	0.006\\
0.0499999999999999	0.006\\
0.0599999999999999	0.0055\\
0.07	0.003\\
0.08	0.0085\\
0.09	0.008\\
0.1	0.007\\
0.11	0.0055\\
0.12	0.008\\
0.13	0.008\\
0.14	0.0085\\
0.15	0.003\\
0.16	0.0085\\
0.17	0.0055\\
0.18	0.005\\
0.19	0.0055\\
0.2	0.004\\
0.21	0.005\\
0.22	0.007\\
0.23	0.002\\
0.24	0.0035\\
0.25	0.006\\
0.26	0.005\\
0.27	0.006\\
0.28	0.0055\\
0.29	0.006\\
0.3	0.0055\\
0.31	0.006\\
0.32	0.0045\\
0.33	0.0065\\
0.34	0.0075\\
0.35	0.0045\\
0.36	0.0055\\
0.37	0.0065\\
0.38	0.0035\\
0.39	0.0055\\
0.4	0.0035\\
0.41	0.004\\
0.42	0.0045\\
0.43	0.006\\
0.44	0.0045\\
0.45	0.0035\\
0.46	0.005\\
0.47	0.004\\
0.48	0.006\\
0.49	0.003\\
0.5	0.0065\\
0.51	0.007\\
0.52	0.0045\\
0.53	0.0065\\
0.54	0.0055\\
0.55	0.0065\\
0.56	0.0055\\
0.57	0.0055\\
0.58	0.006\\
0.59	0.004\\
0.6	0.006\\
0.61	0.003\\
0.62	0.0075\\
0.63	0.0045\\
0.64	0.007\\
0.65	0.0055\\
0.66	0.005\\
0.67	0.0045\\
0.68	0.007\\
0.69	0.0005\\
0.7	0.004\\
0.71	0.008\\
0.72	0.0055\\
0.73	0.006\\
0.74	0.006\\
0.75	0.008\\
0.76	0.0045\\
0.77	0.007\\
0.78	0.006\\
0.79	0.0045\\
0.8	0.0035\\
0.81	0.0045\\
0.82	0.005\\
0.83	0.0065\\
0.84	0.006\\
0.85	0.0075\\
0.86	0.0045\\
0.87	0.007\\
0.88	0.004\\
0.89	0.0035\\
0.9	0.005\\
0.91	0.0045\\
0.92	0.0055\\
0.93	0.0055\\
0.94	0.0035\\
0.95	0.009\\
0.96	0.0065\\
0.97	0.004\\
0.98	0.006\\
0.99	0.0075\\
};

\addplot [color=blue]
  table[row sep=crcr]{%
-0.99	0.3955\\
-0.98	0.2965\\
-0.97	0.217\\
-0.96	0.1595\\
-0.95	0.138\\
-0.94	0.1105\\
-0.93	0.093\\
-0.92	0.094\\
-0.91	0.082\\
-0.9	0.087\\
-0.89	0.076\\
-0.88	0.0845\\
-0.87	0.07\\
-0.86	0.0795\\
-0.85	0.0795\\
-0.84	0.0695\\
-0.83	0.079\\
-0.82	0.068\\
-0.81	0.0755\\
-0.8	0.075\\
-0.79	0.0735\\
-0.78	0.0595\\
-0.77	0.06\\
-0.76	0.067\\
-0.75	0.067\\
-0.74	0.0605\\
-0.73	0.0605\\
-0.72	0.067\\
-0.71	0.0535\\
-0.7	0.0635\\
-0.69	0.06\\
-0.68	0.067\\
-0.67	0.064\\
-0.66	0.067\\
-0.65	0.056\\
-0.64	0.057\\
-0.63	0.058\\
-0.62	0.0575\\
-0.61	0.0505\\
-0.6	0.0645\\
-0.59	0.057\\
-0.58	0.055\\
-0.57	0.0535\\
-0.56	0.0695\\
-0.55	0.054\\
-0.54	0.056\\
-0.53	0.061\\
-0.52	0.049\\
-0.51	0.0515\\
-0.5	0.0545\\
-0.49	0.0655\\
-0.48	0.055\\
-0.47	0.0595\\
-0.46	0.0485\\
-0.45	0.052\\
-0.44	0.0495\\
-0.43	0.0515\\
-0.42	0.056\\
-0.41	0.054\\
-0.4	0.056\\
-0.39	0.05\\
-0.38	0.0615\\
-0.37	0.051\\
-0.36	0.0595\\
-0.35	0.0515\\
-0.34	0.0565\\
-0.33	0.0525\\
-0.32	0.0545\\
-0.31	0.048\\
-0.3	0.0465\\
-0.29	0.0505\\
-0.28	0.0545\\
-0.27	0.052\\
-0.26	0.0525\\
-0.25	0.0625\\
-0.24	0.045\\
-0.23	0.0495\\
-0.22	0.0555\\
-0.21	0.0545\\
-0.2	0.043\\
-0.19	0.0535\\
-0.18	0.043\\
-0.17	0.0465\\
-0.16	0.047\\
-0.15	0.0505\\
-0.14	0.039\\
-0.13	0.0455\\
-0.12	0.0435\\
-0.11	0.044\\
-0.1	0.046\\
-0.09	0.041\\
-0.08	0.0525\\
-0.07	0.0535\\
-0.0599999999999999	0.044\\
-0.0499999999999999	0.045\\
-0.0399999999999999	0.0405\\
-0.03	0.045\\
-0.02	0.038\\
-0.01	0.039\\
0	0.045\\
0.01	0.0385\\
0.02	0.041\\
0.03	0.0445\\
0.0399999999999999	0.0445\\
0.0499999999999999	0.0475\\
0.0599999999999999	0.045\\
0.07	0.041\\
0.08	0.039\\
0.09	0.0445\\
0.1	0.0485\\
0.11	0.0375\\
0.12	0.045\\
0.13	0.037\\
0.14	0.0475\\
0.15	0.035\\
0.16	0.042\\
0.17	0.048\\
0.18	0.043\\
0.19	0.0375\\
0.2	0.038\\
0.21	0.039\\
0.22	0.041\\
0.23	0.029\\
0.24	0.044\\
0.25	0.045\\
0.26	0.038\\
0.27	0.04\\
0.28	0.038\\
0.29	0.0425\\
0.3	0.0425\\
0.31	0.042\\
0.32	0.0385\\
0.33	0.038\\
0.34	0.0445\\
0.35	0.037\\
0.36	0.039\\
0.37	0.0395\\
0.38	0.0385\\
0.39	0.0335\\
0.4	0.0395\\
0.41	0.033\\
0.42	0.041\\
0.43	0.0395\\
0.44	0.039\\
0.45	0.03\\
0.46	0.042\\
0.47	0.0335\\
0.48	0.04\\
0.49	0.037\\
0.5	0.037\\
0.51	0.038\\
0.52	0.046\\
0.53	0.029\\
0.54	0.044\\
0.55	0.0275\\
0.56	0.0395\\
0.57	0.029\\
0.58	0.039\\
0.59	0.032\\
0.6	0.0315\\
0.61	0.031\\
0.62	0.0345\\
0.63	0.0335\\
0.64	0.029\\
0.65	0.034\\
0.66	0.033\\
0.67	0.0325\\
0.68	0.0315\\
0.69	0.0285\\
0.7	0.034\\
0.71	0.027\\
0.72	0.029\\
0.73	0.025\\
0.74	0.0295\\
0.75	0.0275\\
0.76	0.0285\\
0.77	0.022\\
0.78	0.027\\
0.79	0.0305\\
0.8	0.0225\\
0.81	0.022\\
0.82	0.0265\\
0.83	0.029\\
0.84	0.0265\\
0.85	0.024\\
0.86	0.029\\
0.87	0.0235\\
0.88	0.0245\\
0.89	0.028\\
0.9	0.0215\\
0.91	0.0215\\
0.92	0.0175\\
0.93	0.0205\\
0.94	0.021\\
0.95	0.0175\\
0.96	0.021\\
0.97	0.0185\\
0.98	0.0085\\
0.99	0.005\\
};

\addplot [color=red, line width=1.0pt]
  table[row sep=crcr]{%
-0.99	0.05\\
0.99	0.05\\
};

\end{axis}
\end{tikzpicture}% 
		& 
		\begin{tikzpicture}[baseline]

\begin{axis}[%
width=\fwidth,
height=\fheight,
scale only axis,
xmin=-1,
xmax=1,
xlabel={$\rho$},
ymin=0,
ymax=0.4,
axis x line*=bottom,
axis y line*=left
]
\addplot [color=black, dashed]
  table[row sep=crcr]{%
-0.99	0.0055\\
-0.98	0.0065\\
-0.97	0.0055\\
-0.96	0.0035\\
-0.95	0.0065\\
-0.94	0.006\\
-0.93	0.0055\\
-0.92	0.007\\
-0.91	0.004\\
-0.9	0.0055\\
-0.89	0.0055\\
-0.88	0.003\\
-0.87	0.0055\\
-0.86	0.003\\
-0.85	0.005\\
-0.84	0.0065\\
-0.83	0.005\\
-0.82	0.0025\\
-0.81	0.0055\\
-0.8	0.005\\
-0.79	0.007\\
-0.78	0.004\\
-0.77	0.004\\
-0.76	0.0055\\
-0.75	0.0065\\
-0.74	0.005\\
-0.73	0.004\\
-0.72	0.004\\
-0.71	0.0025\\
-0.7	0.004\\
-0.69	0.0065\\
-0.68	0.0045\\
-0.67	0.0035\\
-0.66	0.005\\
-0.65	0.005\\
-0.64	0.0055\\
-0.63	0.004\\
-0.62	0.0065\\
-0.61	0.0055\\
-0.6	0.0035\\
-0.59	0.0065\\
-0.58	0.0085\\
-0.57	0.0055\\
-0.56	0.0045\\
-0.55	0.005\\
-0.54	0.0055\\
-0.53	0.0035\\
-0.52	0.0035\\
-0.51	0.0045\\
-0.5	0.009\\
-0.49	0.005\\
-0.48	0.0045\\
-0.47	0.0035\\
-0.46	0.0025\\
-0.45	0.0045\\
-0.44	0.003\\
-0.43	0.0045\\
-0.42	0.0085\\
-0.41	0.0065\\
-0.4	0.003\\
-0.39	0.0065\\
-0.38	0.0055\\
-0.37	0.003\\
-0.36	0.003\\
-0.35	0.004\\
-0.34	0.004\\
-0.33	0.004\\
-0.32	0.0075\\
-0.31	0.0055\\
-0.3	0.0045\\
-0.29	0.0025\\
-0.28	0.004\\
-0.27	0.0045\\
-0.26	0.006\\
-0.25	0.0055\\
-0.24	0.0055\\
-0.23	0.0075\\
-0.22	0.0055\\
-0.21	0.005\\
-0.2	0.007\\
-0.19	0.0065\\
-0.18	0.006\\
-0.17	0.0105\\
-0.16	0.003\\
-0.15	0.005\\
-0.14	0.0045\\
-0.13	0.0055\\
-0.12	0.0055\\
-0.11	0.006\\
-0.1	0.0045\\
-0.09	0.005\\
-0.08	0.0065\\
-0.07	0.005\\
-0.0599999999999999	0.006\\
-0.0499999999999999	0.0035\\
-0.0399999999999999	0.01\\
-0.03	0.008\\
-0.02	0.006\\
-0.01	0.0065\\
0	0.003\\
0.01	0.006\\
0.02	0.005\\
0.03	0.007\\
0.0399999999999999	0.0085\\
0.0499999999999999	0.0055\\
0.0599999999999999	0.004\\
0.07	0.0065\\
0.08	0.003\\
0.09	0.006\\
0.1	0.006\\
0.11	0.006\\
0.12	0.006\\
0.13	0.0065\\
0.14	0.0025\\
0.15	0.006\\
0.16	0.008\\
0.17	0.008\\
0.18	0.0045\\
0.19	0.006\\
0.2	0.0005\\
0.21	0.0065\\
0.22	0.007\\
0.23	0.0055\\
0.24	0.004\\
0.25	0.009\\
0.26	0.005\\
0.27	0.006\\
0.28	0.005\\
0.29	0.005\\
0.3	0.007\\
0.31	0.004\\
0.32	0.005\\
0.33	0.0055\\
0.34	0.0065\\
0.35	0.0055\\
0.36	0.0065\\
0.37	0.0065\\
0.38	0.0045\\
0.39	0.003\\
0.4	0.0075\\
0.41	0.006\\
0.42	0.005\\
0.43	0.0055\\
0.44	0.0065\\
0.45	0.005\\
0.46	0.0035\\
0.47	0.0035\\
0.48	0.009\\
0.49	0.005\\
0.5	0.006\\
0.51	0.0035\\
0.52	0.0045\\
0.53	0.005\\
0.54	0.004\\
0.55	0.0075\\
0.56	0.0045\\
0.57	0.0035\\
0.58	0.0045\\
0.59	0.006\\
0.6	0.007\\
0.61	0.0075\\
0.62	0.0055\\
0.63	0.0065\\
0.64	0.0045\\
0.65	0.0085\\
0.66	0.005\\
0.67	0.005\\
0.68	0.006\\
0.69	0.005\\
0.7	0.008\\
0.71	0.0065\\
0.72	0.0035\\
0.73	0.006\\
0.74	0.0035\\
0.75	0.0065\\
0.76	0.0075\\
0.77	0.0065\\
0.78	0.0065\\
0.79	0.002\\
0.8	0.0085\\
0.81	0.007\\
0.82	0.0065\\
0.83	0.007\\
0.84	0.006\\
0.85	0.0045\\
0.86	0.007\\
0.87	0.0045\\
0.88	0.006\\
0.89	0.007\\
0.9	0.006\\
0.91	0.008\\
0.92	0.0025\\
0.93	0.0045\\
0.94	0.004\\
0.95	0.0095\\
0.96	0.006\\
0.97	0.005\\
0.98	0.0075\\
0.99	0.0045\\
};

\addplot [color=blue]
  table[row sep=crcr]{%
-0.99	0.362\\
-0.98	0.2595\\
-0.97	0.199\\
-0.96	0.1445\\
-0.95	0.129\\
-0.94	0.122\\
-0.93	0.112\\
-0.92	0.087\\
-0.91	0.0965\\
-0.9	0.08\\
-0.89	0.087\\
-0.88	0.0905\\
-0.87	0.0905\\
-0.86	0.085\\
-0.85	0.0865\\
-0.84	0.075\\
-0.83	0.081\\
-0.82	0.0735\\
-0.81	0.065\\
-0.8	0.07\\
-0.79	0.078\\
-0.78	0.0815\\
-0.77	0.0745\\
-0.76	0.067\\
-0.75	0.0805\\
-0.74	0.0715\\
-0.73	0.0625\\
-0.72	0.0755\\
-0.71	0.0635\\
-0.7	0.061\\
-0.69	0.073\\
-0.68	0.0745\\
-0.67	0.062\\
-0.66	0.0625\\
-0.65	0.0675\\
-0.64	0.0585\\
-0.63	0.0665\\
-0.62	0.0685\\
-0.61	0.066\\
-0.6	0.0575\\
-0.59	0.0705\\
-0.58	0.0555\\
-0.57	0.064\\
-0.56	0.074\\
-0.55	0.0665\\
-0.54	0.059\\
-0.53	0.062\\
-0.52	0.0575\\
-0.51	0.065\\
-0.5	0.061\\
-0.49	0.056\\
-0.48	0.0635\\
-0.47	0.0575\\
-0.46	0.0665\\
-0.45	0.071\\
-0.44	0.0535\\
-0.43	0.0645\\
-0.42	0.054\\
-0.41	0.0585\\
-0.4	0.053\\
-0.39	0.0545\\
-0.38	0.052\\
-0.37	0.0475\\
-0.36	0.0565\\
-0.35	0.059\\
-0.34	0.0605\\
-0.33	0.0505\\
-0.32	0.0585\\
-0.31	0.049\\
-0.3	0.0565\\
-0.29	0.054\\
-0.28	0.0555\\
-0.27	0.054\\
-0.26	0.0515\\
-0.25	0.066\\
-0.24	0.061\\
-0.23	0.0485\\
-0.22	0.054\\
-0.21	0.056\\
-0.2	0.05\\
-0.19	0.0595\\
-0.18	0.0595\\
-0.17	0.064\\
-0.16	0.059\\
-0.15	0.059\\
-0.14	0.0515\\
-0.13	0.054\\
-0.12	0.057\\
-0.11	0.047\\
-0.1	0.056\\
-0.09	0.052\\
-0.08	0.049\\
-0.07	0.0575\\
-0.0599999999999999	0.051\\
-0.0499999999999999	0.0535\\
-0.0399999999999999	0.063\\
-0.03	0.051\\
-0.02	0.042\\
-0.01	0.049\\
0	0.0435\\
0.01	0.045\\
0.02	0.051\\
0.03	0.0535\\
0.0399999999999999	0.0515\\
0.0499999999999999	0.0525\\
0.0599999999999999	0.045\\
0.07	0.049\\
0.08	0.0505\\
0.09	0.058\\
0.1	0.0475\\
0.11	0.043\\
0.12	0.0465\\
0.13	0.044\\
0.14	0.0505\\
0.15	0.051\\
0.16	0.053\\
0.17	0.0525\\
0.18	0.043\\
0.19	0.0505\\
0.2	0.0415\\
0.21	0.043\\
0.22	0.0455\\
0.23	0.042\\
0.24	0.0455\\
0.25	0.052\\
0.26	0.0445\\
0.27	0.0415\\
0.28	0.0435\\
0.29	0.0405\\
0.3	0.038\\
0.31	0.0515\\
0.32	0.043\\
0.33	0.036\\
0.34	0.0515\\
0.35	0.0405\\
0.36	0.049\\
0.37	0.0415\\
0.38	0.04\\
0.39	0.037\\
0.4	0.05\\
0.41	0.0425\\
0.42	0.0355\\
0.43	0.042\\
0.44	0.0425\\
0.45	0.046\\
0.46	0.038\\
0.47	0.033\\
0.48	0.041\\
0.49	0.042\\
0.5	0.045\\
0.51	0.044\\
0.52	0.0375\\
0.53	0.0455\\
0.54	0.041\\
0.55	0.034\\
0.56	0.033\\
0.57	0.034\\
0.58	0.038\\
0.59	0.0395\\
0.6	0.03\\
0.61	0.04\\
0.62	0.0335\\
0.63	0.045\\
0.64	0.032\\
0.65	0.036\\
0.66	0.0335\\
0.67	0.0365\\
0.68	0.0325\\
0.69	0.028\\
0.7	0.0345\\
0.71	0.0335\\
0.72	0.039\\
0.73	0.041\\
0.74	0.035\\
0.75	0.0285\\
0.76	0.036\\
0.77	0.0255\\
0.78	0.0375\\
0.79	0.0305\\
0.8	0.028\\
0.81	0.029\\
0.82	0.024\\
0.83	0.035\\
0.84	0.0315\\
0.85	0.0255\\
0.86	0.0315\\
0.87	0.0285\\
0.88	0.02\\
0.89	0.0255\\
0.9	0.027\\
0.91	0.0255\\
0.92	0.0235\\
0.93	0.021\\
0.94	0.0185\\
0.95	0.0205\\
0.96	0.0155\\
0.97	0.014\\
0.98	0.015\\
0.99	0.0025\\
};

\addplot [color=red, line width=1.0pt]
  table[row sep=crcr]{%
-0.99	0.05\\
0.99	0.05\\
};

\end{axis}
\end{tikzpicture}% 
		&
		\begin{tikzpicture}[baseline]

\begin{axis}[%
width=\fwidth,
height=\fheight,
scale only axis,
xmin=-1,
xmax=1,
xlabel={$\rho$},
ymin=0,
ymax=0.4,
axis x line*=bottom,
axis y line*=left
]
\addplot [color=black, dashed]
  table[row sep=crcr]{%
-0.99	0.0055\\
-0.98	0.0025\\
-0.97	0.002\\
-0.96	0.006\\
-0.95	0.0035\\
-0.94	0.005\\
-0.93	0.007\\
-0.92	0.0035\\
-0.91	0.006\\
-0.9	0.0035\\
-0.89	0.0065\\
-0.88	0.003\\
-0.87	0.008\\
-0.86	0.0065\\
-0.85	0.0035\\
-0.84	0.0055\\
-0.83	0.0065\\
-0.82	0.0065\\
-0.81	0.006\\
-0.8	0.008\\
-0.79	0.0035\\
-0.78	0.0065\\
-0.77	0.005\\
-0.76	0.008\\
-0.75	0.005\\
-0.74	0.0035\\
-0.73	0.003\\
-0.72	0.006\\
-0.71	0.006\\
-0.7	0.002\\
-0.69	0.0035\\
-0.68	0.0035\\
-0.67	0.0035\\
-0.66	0.008\\
-0.65	0.004\\
-0.64	0.008\\
-0.63	0.0035\\
-0.62	0.007\\
-0.61	0.007\\
-0.6	0.002\\
-0.59	0.005\\
-0.58	0.0065\\
-0.57	0.006\\
-0.56	0.0075\\
-0.55	0.0095\\
-0.54	0.0045\\
-0.53	0.005\\
-0.52	0.0035\\
-0.51	0.0045\\
-0.5	0.0045\\
-0.49	0.004\\
-0.48	0.0045\\
-0.47	0.006\\
-0.46	0.006\\
-0.45	0.0055\\
-0.44	0.007\\
-0.43	0.0035\\
-0.42	0.005\\
-0.41	0.0035\\
-0.4	0.0065\\
-0.39	0.0065\\
-0.38	0.0055\\
-0.37	0.003\\
-0.36	0.0035\\
-0.35	0.0075\\
-0.34	0.0055\\
-0.33	0.0095\\
-0.32	0.0055\\
-0.31	0.0055\\
-0.3	0.005\\
-0.29	0.0035\\
-0.28	0.0065\\
-0.27	0.006\\
-0.26	0.006\\
-0.25	0.005\\
-0.24	0.0055\\
-0.23	0.0045\\
-0.22	0.005\\
-0.21	0.003\\
-0.2	0.0045\\
-0.19	0.0005\\
-0.18	0.0075\\
-0.17	0.005\\
-0.16	0.007\\
-0.15	0.003\\
-0.14	0.0035\\
-0.13	0.005\\
-0.12	0.0035\\
-0.11	0.004\\
-0.1	0.005\\
-0.09	0.0025\\
-0.08	0.0055\\
-0.07	0.0035\\
-0.0599999999999999	0.0035\\
-0.0499999999999999	0.0035\\
-0.0399999999999999	0.004\\
-0.03	0.0065\\
-0.02	0.0025\\
-0.01	0.0065\\
0	0.005\\
0.01	0.0025\\
0.02	0.0015\\
0.03	0.006\\
0.0399999999999999	0.0055\\
0.0499999999999999	0.003\\
0.0599999999999999	0.0045\\
0.07	0.0035\\
0.08	0.005\\
0.09	0.0035\\
0.1	0.005\\
0.11	0.005\\
0.12	0.004\\
0.13	0.0045\\
0.14	0.0075\\
0.15	0.0025\\
0.16	0.003\\
0.17	0.002\\
0.18	0.003\\
0.19	0.006\\
0.2	0.003\\
0.21	0.004\\
0.22	0.0035\\
0.23	0.005\\
0.24	0.007\\
0.25	0.0045\\
0.26	0.0045\\
0.27	0.0065\\
0.28	0.0075\\
0.29	0.004\\
0.3	0.0045\\
0.31	0.0045\\
0.32	0.003\\
0.33	0.004\\
0.34	0.007\\
0.35	0.0055\\
0.36	0.0035\\
0.37	0.0045\\
0.38	0.003\\
0.39	0.003\\
0.4	0.0055\\
0.41	0.0075\\
0.42	0.0035\\
0.43	0.0065\\
0.44	0.007\\
0.45	0.0025\\
0.46	0.0065\\
0.47	0.005\\
0.48	0.005\\
0.49	0.008\\
0.5	0.005\\
0.51	0.005\\
0.52	0.0035\\
0.53	0.004\\
0.54	0.0035\\
0.55	0.0045\\
0.56	0.006\\
0.57	0.0045\\
0.58	0.006\\
0.59	0.005\\
0.6	0.0035\\
0.61	0.002\\
0.62	0.005\\
0.63	0.005\\
0.64	0.007\\
0.65	0.0035\\
0.66	0.006\\
0.67	0.007\\
0.68	0.004\\
0.69	0.005\\
0.7	0.0045\\
0.71	0.005\\
0.72	0.0055\\
0.73	0.005\\
0.74	0.0035\\
0.75	0.005\\
0.76	0.005\\
0.77	0.003\\
0.78	0.0035\\
0.79	0.005\\
0.8	0.0045\\
0.81	0.005\\
0.82	0.004\\
0.83	0.0085\\
0.84	0.0035\\
0.85	0.0055\\
0.86	0.006\\
0.87	0.008\\
0.88	0.003\\
0.89	0.0065\\
0.9	0.005\\
0.91	0.0045\\
0.92	0.006\\
0.93	0.005\\
0.94	0.003\\
0.95	0.0045\\
0.96	0.006\\
0.97	0.0055\\
0.98	0.0045\\
0.99	0.004\\
};
\label{disjointlarge}

\addplot [color=blue]
  table[row sep=crcr]{%
-0.99	0.3675\\
-0.98	0.266\\
-0.97	0.181\\
-0.96	0.1465\\
-0.95	0.128\\
-0.94	0.1075\\
-0.93	0.1195\\
-0.92	0.1\\
-0.91	0.097\\
-0.9	0.0985\\
-0.89	0.1005\\
-0.88	0.0925\\
-0.87	0.084\\
-0.86	0.0905\\
-0.85	0.0915\\
-0.84	0.089\\
-0.83	0.091\\
-0.82	0.077\\
-0.81	0.0775\\
-0.8	0.0795\\
-0.79	0.077\\
-0.78	0.068\\
-0.77	0.071\\
-0.76	0.0845\\
-0.75	0.0775\\
-0.74	0.0765\\
-0.73	0.073\\
-0.72	0.0875\\
-0.71	0.075\\
-0.7	0.059\\
-0.69	0.0785\\
-0.68	0.074\\
-0.67	0.069\\
-0.66	0.0715\\
-0.65	0.078\\
-0.64	0.0625\\
-0.63	0.0705\\
-0.62	0.0695\\
-0.61	0.0775\\
-0.6	0.081\\
-0.59	0.068\\
-0.58	0.0585\\
-0.57	0.0685\\
-0.56	0.061\\
-0.55	0.0765\\
-0.54	0.072\\
-0.53	0.0725\\
-0.52	0.0595\\
-0.51	0.067\\
-0.5	0.0685\\
-0.49	0.0545\\
-0.48	0.0705\\
-0.47	0.0695\\
-0.46	0.064\\
-0.45	0.0715\\
-0.44	0.068\\
-0.43	0.0655\\
-0.42	0.0665\\
-0.41	0.0635\\
-0.4	0.056\\
-0.39	0.061\\
-0.38	0.067\\
-0.37	0.0535\\
-0.36	0.0565\\
-0.35	0.053\\
-0.34	0.0645\\
-0.33	0.065\\
-0.32	0.0605\\
-0.31	0.0615\\
-0.3	0.0625\\
-0.29	0.0565\\
-0.28	0.057\\
-0.27	0.0655\\
-0.26	0.058\\
-0.25	0.064\\
-0.24	0.052\\
-0.23	0.0625\\
-0.22	0.06\\
-0.21	0.052\\
-0.2	0.054\\
-0.19	0.069\\
-0.18	0.064\\
-0.17	0.0685\\
-0.16	0.0635\\
-0.15	0.0555\\
-0.14	0.0605\\
-0.13	0.052\\
-0.12	0.066\\
-0.11	0.063\\
-0.1	0.0625\\
-0.09	0.0505\\
-0.08	0.054\\
-0.07	0.059\\
-0.0599999999999999	0.056\\
-0.0499999999999999	0.0565\\
-0.0399999999999999	0.051\\
-0.03	0.051\\
-0.02	0.0385\\
-0.01	0.055\\
0	0.049\\
0.01	0.0565\\
0.02	0.053\\
0.03	0.057\\
0.0399999999999999	0.05\\
0.0499999999999999	0.0595\\
0.0599999999999999	0.052\\
0.07	0.0495\\
0.08	0.058\\
0.09	0.0485\\
0.1	0.055\\
0.11	0.0415\\
0.12	0.0605\\
0.13	0.053\\
0.14	0.0495\\
0.15	0.0615\\
0.16	0.0625\\
0.17	0.0615\\
0.18	0.051\\
0.19	0.0605\\
0.2	0.0535\\
0.21	0.048\\
0.22	0.0465\\
0.23	0.0535\\
0.24	0.0535\\
0.25	0.0475\\
0.26	0.0525\\
0.27	0.045\\
0.28	0.0495\\
0.29	0.0545\\
0.3	0.0475\\
0.31	0.055\\
0.32	0.0465\\
0.33	0.053\\
0.34	0.053\\
0.35	0.0495\\
0.36	0.044\\
0.37	0.0435\\
0.38	0.0555\\
0.39	0.041\\
0.4	0.0395\\
0.41	0.0545\\
0.42	0.041\\
0.43	0.0465\\
0.44	0.045\\
0.45	0.052\\
0.46	0.045\\
0.47	0.0465\\
0.48	0.0475\\
0.49	0.042\\
0.5	0.0545\\
0.51	0.04\\
0.52	0.043\\
0.53	0.045\\
0.54	0.0415\\
0.55	0.04\\
0.56	0.045\\
0.57	0.036\\
0.58	0.038\\
0.59	0.0445\\
0.6	0.041\\
0.61	0.044\\
0.62	0.043\\
0.63	0.044\\
0.64	0.035\\
0.65	0.0415\\
0.66	0.0445\\
0.67	0.0405\\
0.68	0.036\\
0.69	0.0415\\
0.7	0.0365\\
0.71	0.036\\
0.72	0.0335\\
0.73	0.048\\
0.74	0.042\\
0.75	0.0345\\
0.76	0.031\\
0.77	0.04\\
0.78	0.0325\\
0.79	0.038\\
0.8	0.0385\\
0.81	0.0335\\
0.82	0.0285\\
0.83	0.037\\
0.84	0.027\\
0.85	0.0345\\
0.86	0.0325\\
0.87	0.03\\
0.88	0.026\\
0.89	0.02\\
0.9	0.0295\\
0.91	0.025\\
0.92	0.029\\
0.93	0.0205\\
0.94	0.0285\\
0.95	0.0245\\
0.96	0.024\\
0.97	0.0175\\
0.98	0.0125\\
0.99	0.012\\
};
\label{fulllarge}

\addplot [color=red, line width=1.0pt]
  table[row sep=crcr]{%
-0.99	0.05\\
0.99	0.05\\
};
\label{level}

\end{axis}
\end{tikzpicture}% 
		\\
	\end{tabular}
	\caption{Type I error of the tests $\Phi_n^{\mathrm{a}}$ (\ref{fulllarge}) and $\Phi_n^{\mathrm{d}}$ (\ref{disjointlarge}) for the AR(1) case vs. $\rho$ ($x$-axis) simulated by $2000$ Monte Carlo simulations with the nominal type I error $\alpha = 0.05$ (\ref{level}) in three different situations: small sample size $\lambda_n = 0.1$, $n = 829$ (left), medium sample size $\lambda_n = 0.05$, $n = 2157$ (middle) and large sample size $\lambda_n = 0.025$, $n = 5312$ (right).} 
	\label{fig:level}
\end{figure}

Next we computed the finite sample type II error in all three sample size situations for $\rho \in \left\{-0.99,-0.98,...,0.99\right\}$ and $\Delta_n \in \left\{0.01,0.02,...,0.5\right\}$. The corresponding results are shown in Figures \ref{fig:small}--\ref{fig:large}. We also depict the contour line of equation \eqref{eq:rhoDelta} for a comparison and find a remarkably good agreement with the contour lines of the power function already in the small sample regime, which strongly supports the finite sample validity of our asymptotic theory. Finally, we conclude that detection becomes easier for a larger number of bumps.

\begin{figure}[!htbp]
	\setlength\fheight{2.8cm} \setlength\fwidth{2.8cm} 
	\centering
	\begin{tabular}{lll}
		\qquad One bump
		&
		\qquad Two bumps
		&
		\qquad Five bumps
		\\
		\begin{tikzpicture}[baseline]

\begin{axis}[%
width=\fwidth,
height=\fheight,
scale only axis,
axis on top,
xmin=-1,
xmax=1,
ymin=0,
ymax=0.5
%ylabel={$\Delta$},
%colormap/hot2,
%colorbar,
%point meta min=0,
%point meta max=1
]
\addplot [forget plot] graphics [xmin=-1, xmax=1, ymin=0, ymax=0.5] {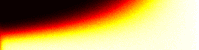};

\addplot [domain=-1:0.55, samples=150,dashed,thick]{0.236*1/(1-x)};
\end{axis}
\end{tikzpicture}% 
		& 
		\begin{tikzpicture}[baseline]

\begin{axis}[%
width=\fwidth,
height=\fheight,
scale only axis,
axis on top,
xmin=-1,
xmax=1,
ymin=0,
ymax=0.5
%ylabel={$\Delta$},
%colormap/hot2,
%colorbar,
%point meta min=0,
%point meta max=1
]
\addplot [forget plot] graphics [xmin=-1, xmax=1, ymin=0, ymax=0.5] {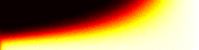};
\end{axis}
\end{tikzpicture}% 
		&
		\begin{tikzpicture}[baseline]

\begin{axis}[%
width=\fwidth,
height=\fheight,
scale only axis,
axis on top,
xmin=-1,
xmax=1,
ymin=0,
ymax=0.5,
colormap/hot2,
colorbar,
point meta min=0,
point meta max=1
]
\addplot [forget plot] graphics [xmin=-1, xmax=1, ymin=0, ymax=0.5] {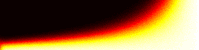};
\end{axis}
\end{tikzpicture}% 
		\\
		\begin{tikzpicture}[baseline]

\begin{axis}[%
width=\fwidth,
height=\fheight,
scale only axis,
axis on top,
xmin=-1,
xmax=1,
ymin=0,
ymax=0.5
%ylabel={$\Delta$},
%colormap/hot2,
%colorbar,
%point meta min=0,
%point meta max=1
]
\addplot [forget plot] graphics [xmin=-1, xmax=1, ymin=0, ymax=0.5] {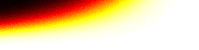};

\addplot [domain=-1:0.55, samples=150,dashed,thick]{0.236*1/(1-x)};
\end{axis}
\end{tikzpicture}% 
		& 
		\begin{tikzpicture}[baseline]

\begin{axis}[%
width=\fwidth,
height=\fheight,
scale only axis,
axis on top,
xmin=-1,
xmax=1,
ymin=0,
ymax=0.5
%ylabel={$\Delta$},
%colormap/hot2,
%colorbar,
%point meta min=0,
%point meta max=1
]
\addplot [forget plot] graphics [xmin=-1, xmax=1, ymin=0, ymax=0.5] {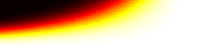};
\end{axis}
\end{tikzpicture}% 
		&
		\begin{tikzpicture}[baseline]

\begin{axis}[%
width=\fwidth,
height=\fheight,
scale only axis,
axis on top,
xmin=-1,
xmax=1,
ymin=0,
ymax=0.5,
colormap/hot2,
colorbar,
point meta min=0,
point meta max=1
]
\addplot [forget plot] graphics [xmin=-1, xmax=1, ymin=0, ymax=0.5] {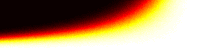};
\end{axis}
\end{tikzpicture}% 
		\\
	\end{tabular}
	\caption{Type II error in the small sample regime $\lambda_n = 0.1$, $n = 829$ (top row: $\Phi_n^{\mathrm{a}}$; bottom row: $\Phi_n^{\mathrm{d}}$) for the AR(1) case for $\rho$ ($x$-axis) vs. $\Delta_n$ ($y$-axis) with one bump (left column) together with the contour line of the detection boundary equation \eqref{eq:rhoDelta}, two bumps (middle column) and five bumps (right column), each simulated by $2000$ Monte Carlo simulations.}
	\label{fig:small}
\end{figure}

\begin{figure}[!htbp]
	\setlength\fheight{2.8cm} \setlength\fwidth{2.8cm} 
	\centering
	\begin{tabular}{lll}
		\qquad One bump
		&
		\qquad Two bumps
		&
		\qquad Five bumps
		\\
		\begin{tikzpicture}[baseline]

\begin{axis}[%
width=\fwidth,
height=\fheight,
scale only axis,
axis on top,
xmin=-1,
xmax=1,
ymin=0,
ymax=0.5
%ylabel={$\Delta$},
%colormap/hot2,
%colorbar,
%point meta min=0,
%point meta max=1
]
\addplot [forget plot] graphics [xmin=-1, xmax=1, ymin=0, ymax=0.5] {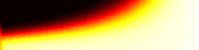};

\addplot [domain=-1:0.55, samples=150,dashed,thick]{0.236*1/(1-x)};
\end{axis}
\end{tikzpicture}% 
		& 
		\begin{tikzpicture}[baseline]

\begin{axis}[%
width=\fwidth,
height=\fheight,
scale only axis,
axis on top,
xmin=-1,
xmax=1,
ymin=0,
ymax=0.5
%ylabel={$\Delta$},
%colormap/hot2,
%colorbar,
%point meta min=0,
%point meta max=1
]
\addplot [forget plot] graphics [xmin=-1, xmax=1, ymin=0, ymax=0.5] {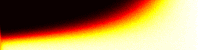};
\end{axis}
\end{tikzpicture}% 
		&
		\begin{tikzpicture}[baseline]

\begin{axis}[%
width=\fwidth,
height=\fheight,
scale only axis,
axis on top,
xmin=-1,
xmax=1,
ymin=0,
ymax=0.5,
colormap/hot2,
colorbar,
point meta min=0,
point meta max=1
]
\addplot [forget plot] graphics [xmin=-1, xmax=1, ymin=0, ymax=0.5] {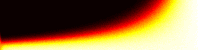};
\end{axis}
\end{tikzpicture}% 
		\\
		\begin{tikzpicture}[baseline]

\begin{axis}[%
width=\fwidth,
height=\fheight,
scale only axis,
axis on top,
xmin=-1,
xmax=1,
ymin=0,
ymax=0.5
%ylabel={$\Delta$},
%colormap/hot2,
%colorbar,
%point meta min=0,
%point meta max=1
]
\addplot [forget plot] graphics [xmin=-1, xmax=1, ymin=0, ymax=0.5] {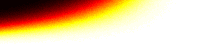};

\addplot [domain=-1:0.55, samples=150,dashed,thick]{0.236*1/(1-x)};
\end{axis}

\end{tikzpicture} 
		& 
		\begin{tikzpicture}[baseline]

\begin{axis}[%
width=\fwidth,
height=\fheight,
scale only axis,
axis on top,
xmin=-1,
xmax=1,
ymin=0,
ymax=0.5
%ylabel={$\Delta$},
%colormap/hot2,
%colorbar,
%point meta min=0,
%point meta max=1
]
\addplot [forget plot] graphics [xmin=-1, xmax=1, ymin=0, ymax=0.5] {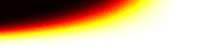};
\end{axis}

\end{tikzpicture} 
		&
		\begin{tikzpicture}[baseline]

\begin{axis}[%
width=\fwidth,
height=\fheight,
scale only axis,
axis on top,
xmin=-1,
xmax=1,
ymin=0,
ymax=0.5,
colormap/hot2,
colorbar,
point meta min=0,
point meta max=1
]
\addplot [forget plot] graphics [xmin=-1, xmax=1, ymin=0, ymax=0.5] {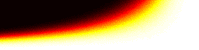};
\end{axis}

\end{tikzpicture} 
		\\
	\end{tabular}
	\caption{Type II error in the  medium sample regime $\lambda_n = 0.05$, $n = 2157$ (top row: $\Phi_n^{\mathrm{a}}$; bottom row: $\Phi_n^{\mathrm{d}}$) for the AR(1) case for $\rho$ ($x$-axis) vs. $\Delta_n$ ($y$-axis) with one bump (left column) together with the contour line of the detection boundary equation \eqref{eq:rhoDelta}, two bumps (middle column) and five bumps (right column), each simulated by $2000$ Monte Carlo simulations.}
	\label{fig:medium}
\end{figure}

\begin{figure}[!htbp]
	\setlength\fheight{2.8cm} \setlength\fwidth{2.8cm} 
	\centering
	\begin{tabular}{lll}
		\qquad One bump
		&
		\qquad Two bumps
		&
		\qquad Five bumps
		\\
		\begin{tikzpicture}[baseline]

\begin{axis}[%
width=\fwidth,
height=\fheight,
scale only axis,
axis on top,
xmin=-1,
xmax=1,
ymin=0,
ymax=0.5
%ylabel={$\Delta$},
%colormap/hot2,
%colorbar,
%point meta min=0,
%point meta max=1
]
\addplot [forget plot] graphics [xmin=-1, xmax=1, ymin=0, ymax=0.5] {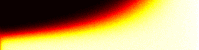};

\addplot [domain=-1:0.55, samples=150,dashed,thick]{0.236*1/(1-x)};
\end{axis}
\end{tikzpicture}% 
		& 
		\begin{tikzpicture}[baseline]

\begin{axis}[%
width=\fwidth,
height=\fheight,
scale only axis,
axis on top,
xmin=-1,
xmax=1,
ymin=0,
ymax=0.5
%ylabel={$\Delta$},
%colormap/hot2,
%colorbar,
%point meta min=0,
%point meta max=1
]
\addplot [forget plot] graphics [xmin=-1, xmax=1, ymin=0, ymax=0.5]  {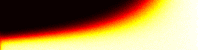};
\end{axis}
\end{tikzpicture}% 
		&
		\begin{tikzpicture}[baseline]

\begin{axis}[%
width=\fwidth,
height=\fheight,
scale only axis,
axis on top,
xmin=-1,
xmax=1,
ymin=0,
ymax=0.5,
colormap/hot2,
colorbar,
point meta min=0,
point meta max=1
]
\addplot [forget plot] graphics [xmin=-1, xmax=1, ymin=0, ymax=0.5] {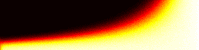};
\end{axis}
\end{tikzpicture}% 
		\\
		\begin{tikzpicture}[baseline]

\begin{axis}[%
width=\fwidth,
height=\fheight,
scale only axis,
axis on top,
xmin=-1,
xmax=1,
ymin=0,
ymax=0.5
%ylabel={$\Delta$},
%colormap/hot2,
%colorbar,
%point meta min=0,
%point meta max=1
]
\addplot [forget plot] graphics [xmin=-1, xmax=1, ymin=0, ymax=0.5] {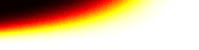};

\addplot [domain=-1:0.55, samples=150,dashed,thick]{0.236*1/(1-x)};
\end{axis}
\end{tikzpicture}% 
		& 
		\begin{tikzpicture}[baseline]

\begin{axis}[%
width=\fwidth,
height=\fheight,
scale only axis,
axis on top,
xmin=-1,
xmax=1,
ymin=0,
ymax=0.5
%ylabel={$\Delta$},
%colormap/hot2,
%colorbar,
%point meta min=0,
%point meta max=1
]
\addplot [forget plot] graphics [xmin=-1, xmax=1, ymin=0, ymax=0.5] {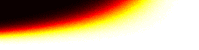};
\end{axis}
\end{tikzpicture}% 
		&
		\begin{tikzpicture}[baseline]

\begin{axis}[%
width=\fwidth,
height=\fheight,
scale only axis,
axis on top,
xmin=-1,
xmax=1,
ymin=0,
ymax=0.5,
colormap/hot2,
colorbar,
point meta min=0,
point meta max=1
]
\addplot [forget plot] graphics [xmin=-1, xmax=1, ymin=0, ymax=0.5] {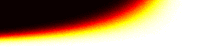};
\end{axis}
\end{tikzpicture}% 
		\\
	\end{tabular}
	\caption{Type II error in the large sample regime $\lambda_n = 0.025$, $n = 5312$ (top row: $\Phi_n^{\mathrm{a}}$; bottom row: $\Phi_n^{\mathrm{d}}$) for the AR(1) case for $\rho$ ($x$-axis) vs. $\Delta_n$ ($y$-axis) with one bump (left column) together with the contour line of the detection boundary equation \eqref{eq:rhoDelta}, two bumps (middle column) and five bumps (right column), each simulated by $2000$ Monte Carlo simulations.}
	\label{fig:large}
\end{figure}

\section{Proofs}\label{sec:proofs}

We begin this section by stating several useful results from various sources that we are going to use in our proofs. The proof of Theorem \ref{thm:mainthm(poly)} will then be split into three parts. We will provide asymtotic upper and lower bounds in subsections \ref{subsec:app_upper_bound} and \ref{subsec:app_lower_bound}, respectively. These results will in fact hold for a wider class of covariance matrices than those allowed by Assumption \ref{cond:P}. Finally, in subsection \ref{subsec:2.1proof}, this will be used to show that the upper and lower bound coincide asymptotically in the setting of Theorem \ref{thm:mainthm(poly)}, and this will yield the desired result.

\subsection{Auxiliary results}

\subsubsection{Weak law of large numbers for arrays of dependent variables}\label{app:wlln}
%This result on the WLNN for arrays of dependent random variables is due to \citet{wanghu14}, and will be necessary for our proof of the lower detection bound.

%The notion of \textit{$h$-integrability with exponent $1$}, stated below, was introduced by \citet{sung08} (derived from the notion of \textit{$h$-integrability concerning an array of weights}, introduced by \citet{cabrera2005}).

\begin{Definition}[\citet{sung08}]\label{def:h_int}
	Let $\lbrace X_{nk}: n\in\N, u_n\leq k\leq v_n\rbrace$ be an array of random variables with $v_n-u_n\to \infty$ as $n\to\infty$. Additionally, let $r>0$, and $(k_n)_{n\in\N}$ be a sequence of positive integers, such that $k_n\to\infty$ as $n\to\infty$.\\
	Let $(h(n))_{n\in\N}$ be a sequence of positive constants, such that $h(n)\nearrow\infty$ as $n\to\infty$. The array $\lbrace X_{nk}: n\in\N, u_n\leq k\leq v_n\rbrace$ is said to be $h$-integrable with exponent $r$ if 
	\[
	\sup_{n\in\N}\frac{1}{k_n}\sum_{k=u_n}^{v_n}\E{}{|X_{nk}|^r}<\infty,
	\quad \text{and} \quad
	\lim_{n\to\infty}\frac{1}{k_n}\sum_{k=u_n}^{v_n}\E{}{|X_{nk}|^r\mathbbm{1}\left\lbrace |X_{nk}|^r>h(n)\right\rbrace}=0.
	\]
\end{Definition}

With this, we have the following.

\begin{Theorem}[\citet{wanghu14}]\label{th:WLLN}
	Let $m$ be a positive integer. Suppose that $\{X_{nk}, u_{n}\le k\le v_n,\ n\ge 1\}$  is an array of non-negative random variables with $\cov (X_{nk},X_{nk})\le 0$ whenever $|j-k|\ge m$, $u_n\le j$, $k\le v_n$, for each $n\ge 1$ and is {$h$}-integrable with exponent $r=1$ for a sequence $k_n\to\infty$ and $h(n)\uparrow \infty$, such that $h(n)/k_n\to 0$ as $n\to\infty$. Then
	$$
	\frac1{k_n}\sum\limits_{k=u_n}^{v_n} (X_{nk}-E{}{X_{nk}}) \to 0
	$$
	in $L_1$ and hence in probability, as $n\to\infty$. 
\end{Theorem}

{
	\begin{Remark}
		In fact, the original theorem in \citet{wanghu14} is slightly stronger, but Theorem \ref{th:WLLN} as stated above is sufficient for our purposes.
	\end{Remark}
}

\begin{Remark}\label{rem:WLLN}
	We can relax the condition $\cov(X_{nj},X_{nk})\le 0$ whenever $|j-k|\ge m$, $u_n\le j,k\le v_n$ in Theorem~\ref{th:WLLN} to requiring only that 
	$$
	\limsup_{n\to\infty} \frac1{k_n^2} \sum_{\substack{j,k=u_n\\ |j-k|\ge m}}^{v_n} \cov(X_{nj},X_{nk}) \le 0.
	$$
\end{Remark}

\subsubsection{Decay of precision matrices}\label{app:precision_matrix}

The following result is due to \citet{jaffard1990} and was used in \cite{hj10} as a key tool in the analysis of a higher criticism test for {detection} of sparse signals {observed} in correlated noise.  Here it is stated as it was formulated and proven in \cite{hj10}.

\begin{Lemma}[\citet{hj10}]\label{lm:halljin} Let $\Sigma_n$, $n\geq 1$ be a sequence of $n \times n$ correlation matrices, such that $\Vert \Sigma_n\Vert\geq c>0$, where $\Vert \Sigma_n\Vert$ is the operator norm of $\Sigma_n$ as an operator from $\R^n$ to $\R^n$. If for some constants $\kappa>0, C>0$, 
	\[
	|\Sigma_n(i,j)|\leq C(1+|i-j|)^{-(1+\kappa)},
	\]
	then there is a constant $C'>0$ depending on $\kappa$, $C$, and $c$, such that
	\[
	|\Sigma_n^{-1}(i,j)|\leq C'(1+|i-j|)^{-(1+\kappa)}.
	\]
\end{Lemma}

\subsubsection{Long-run variance of partial sums of a stationary time series}\label{app:ibragimov}
Here we recall the well-known result given in Theorem 28.2.1 {of} \citet{ibragimov_independent_1971}  on the explicit formula for the variance of the sum of $n$ consecutive realizations of a stationary process. We adapt the notation to our case. 

Suppose that $(X_n)_{n\in\mathbb Z}$ is a centered stationary sequence with the autocovariance function $\gamma(h)$, $h\in \mathbb Z$ and the spectral density $f(\nu)$, $\nu\in[-1/2,1/2)$.
Let $S_n=\sum\limits_{i=1}^n X_i$. 
\begin{Theorem}[\citet{ibragimov_independent_1971}]
	The variance of $S_n$ in terms of $\gamma(h)$ and $f(\nu)$ is given by
	\[
	\mathrm{Var} [S_n] =\sum_{|h|<n} (n-|h|) \gamma(h) =\int_{-1/2}^{1/2}\frac{\sin^2(\pi n\nu)}{\sin^2(\pi \nu)} f(\nu)\,d\nu.
	\]
	If the spectral density $f(\nu)$ is continuous at $\nu=0$, then
	$$
	\mathrm{Var} [S_n]=f(0)n +o(n),\quad n\to\infty. 
	$$ 
\end{Theorem}

\subsection{Proofs for Section \ref{sec:overview}}

\subsubsection{Upper detection bound}\label{subsec:app_upper_bound}

For $I\in\mathcal{I}(\lambda_n)$ we define
\[
\sigma_n(I):=\1_I^T\Sigma_n\1_I.
\]

\begin{Lemma}[Upper detection bound]\label{lm:app_upper_bound} 
	Fix $\alpha\in (0,1)$, consider the testing problem \eqref{eq:testing_problem} and suppose that Assumption \ref{cond:I} {and Assumption \ref{cond:P} hold. In addition, assume that the sequence $(\Sigma_n)_{n\in\N}$ of covariance matrices satisfies} 
	\begin{equation}\label{eq:app_upper_bound}
	\Delta_n\inf_{I\in \mathcal{I}(\lambda_n)}\frac{\lfloor n\lambda_n\rfloor}{\sqrt{\sigma_n(I)}} \succsim (\sqrt{2}+\tilde \varepsilon_n)\sqrt{-\log \lambda_n},
	\end{equation}
	as $n\to\infty$, where $(\tilde \varepsilon_n)_{n\in\N}$ is a sequence of real numbers that satisfies $\tilde \varepsilon_n\to 0$ and $\tilde \varepsilon_n\sqrt{-\log \lambda_n}{-\sqrt{\log\log n}}\to\infty$ as $n\to\infty$.
	
	Then the sequence of level $\alpha$ tests $(\Phi_n^{\mathrm{a}})_{n\in\N}$ as in \eqref{eq:Phi_full} with 
	suitably chosen $c_{\alpha, n}$ satisfies $\bar{\alpha}(\Phi_n^{\mathrm{a}})\leq \alpha$ for all $n\in\N$ and $\limsup_{n\to\infty}\bar{\beta}(\Phi_n^{\mathrm{a}})\leq \alpha$.
\end{Lemma}

\begin{proof}
	Define the test statistic
	\[
	T_n(Y)=\sup_{I\in \mathcal{I}(\lambda_n)}\frac{\left|\1_{I}^TY\right|}{\sqrt{\sigma_n(I)}},
	\]
	and recall that $\Phi_n^{\mathrm{a}}(Y):=\mathbbm{1}\{T_n(Y)>c_{\alpha,n}\},$ for some threshold $c_{\alpha,n}$ to be determined later.
	
	{We begin by noting that although the random process \[\left(\frac{\left|\1_{I}^TY\right|}{\sqrt{\sigma_n(I)}}\right)_{I\in \mathcal{I}(\lambda_n)}\] has an infinite index set, it only takes finitely many different values. Thus, we can find a finite representative system $\mathcal{I}_{\text{fin}}(\lambda_n)\subseteq\mathcal{I}(\lambda_n) $, such that for any $I\in\mathcal{I}(\lambda_n)$ there is $I'\in\mathcal{I}_{\text{fin}}(\lambda_n)$, such that
		\[
		\frac{\left|\1_{I}^TY\right|}{\sqrt{\sigma_n(I)}}=\frac{\left|\1_{I'}^TY\right|}{\sqrt{\sigma_n(I')}},
		\]
		which implies that
		\[
		T_n(Y)=\sup_{I\in \mathcal{I}_{\text{fin}}(\lambda_n)}\frac{\left|\1_{I}^TY\right|}{\sqrt{\sigma_n(I)}} 
		%= \sup_{\substack{I\in \mathcal{I}_{\text{fin}}(\lambda_n) \\ s\in\{-1,+1\}}}\frac{s\1_{I}^TY}{\sqrt{\sigma_n(I)}},
		\]
		i.e. $T_n(Y)$ can be written as the supremum over the absolute values of a Gaussian process with a finite index set.} Let $M_n$ such that ${\mathbb{E}_0 T_n(Y)\leq M_n}$. Then, for any $\lambda>0$, it follows that
	\begin{align*}
	\mathbb{P}_0\left(T_n(Y)>\lambda+M_n\right) &\leq \mathbb{P}_0\left(T_n(Y)-{\mathbb{E}_0 T_n(Y)}>\lambda\right)\\
	&\leq \mathbb{P}_0\left({\left|\sup_{I\in \mathcal{I}_{\text{fin}}(\lambda_n) }\frac{|\1_{I}^TY|}{\sqrt{\sigma_n(I)}}-\mathbb{E}_0 \sup_{I\in \mathcal{I}_{\text{fin}}(\lambda_n) }\frac{|\1_{I}^TY|}{\sqrt{\sigma_n(I)}}\right|}>\lambda\right) \leq 2e^{-\frac{\lambda^2}{2}},
	\end{align*}
	where the last inequality follows  the results of \citet{talagrand2006} {and can be found in Theorem 2.1.20 of \cite{ginenickl2016}.} Thus, if we let
	\[
	c_{\alpha,n}=\sqrt{2\log \frac{2}{\alpha}}+M_n,
	\]
	$\Phi_n^{\mathrm{a}}$ has level $\alpha$ for any $n$.

	{In order to find a suitable bound $M_n$ we consider an even coarser finite subset of $\mathcal{I}(\lambda_n)$.} Let
	\[
	\mathcal{C}_n=\left\{\left[\frac{k}{n}, \frac{k}{n}+\lambda_n\right) : 1\leq k\leq \lfloor n(1-\lambda_n)\rfloor\right\}\subseteq \mathcal{I}(\lambda_n).
	\]
	Clearly, $\#\mathcal{C}_n=\lfloor n(1-\lambda_n)\rfloor\leq n<\infty$. {For any $I\in\mathcal{I}(\lambda_n)$ there is $I'\in\mathcal{C}_n$, such that $1_I$ differs from $1_{I'}$ in at most one entry. Thus, it is easy to see that
		\[
		T_n(Y)\leq \sup_{I\in \mathcal{C}_n}\frac{\left|\1_{I}^TY\right|}{\sqrt{\sigma_n(I)}} +\frac{\sup_{1\leq i\leq n} |Y_{i,n}|}{\inf_{I\in\mathcal{I}(\lambda_n)}\sqrt{\sigma_n(I)}},
		\]
		Thus, we can set
		\[
		M_n=\tilde{M}_n+\kappa_n,
		\]
		where
		\[
		\tilde{M}_n= \mathbb{E}_0 \sup_{I\in \mathcal{C}_n}\frac{\left|\1_{I}^TY\right|}{\sqrt{\sigma_n(I)}},
		\]
		and
		\[
		\kappa_n=\mathbb{E}_0\frac{\sup_{1\leq i\leq n} |Y_{i,n}|}{\inf_{I\in\mathcal{I}(\lambda_n)}\sqrt{\sigma_n(I)}}.
		\]
		The latter term is easy to handle: We have
		\[
		\mathbb{E}_0 \sup_{1\leq i\leq n} |Y_{i,n}| \leq \sqrt{2f_0\log (2n)},
		\]
		since $Y_{i,n}$ has variance $f_0$ for any $n$ and $1\leq i\leq n$, and
		\[
		\inf_{I\in\mathcal{I}(\lambda_n)}\sigma_n(I)= n\lambda_nf(0)(1+o(1))
		\]
		by Theorem 18.2.1 of \citet{ibragimov_independent_1971}. Thus,
		\[
		\kappa_n = O\left(\sqrt{\frac{\log n}{n\lambda_n}}\right),
		\]
		and thus, $\kappa_n\to 0$ by Assumption \ref{cond:I}. The next part of the proof will be devoted to computing $\tilde{M}_n$.} Note that under $H_0$, we have
	\[
	\frac{\1_{I}^TY}{\sqrt{\sigma_n(I)}}\sim \mathcal{N}(0,1),
	\]
	for any $I\in \mathcal{C}_n$.  For any $I, I'\in \mathcal{C}_n$, we have
	\begin{align*}
	\left|\frac{\1_{I}^TY}{\sqrt{\sigma_n(I)}}-\frac{\1_{I'}^TY}{\sqrt{\sigma_n(I')}}\right| = \left|\left(\frac{\1_{I}^T}{\sqrt{\sigma_n(I)}}-\frac{\1_{I'}^T}{\sqrt{\sigma_n(I')}}\right)Y\right|=\left(2-2\frac{\1_{I}^T\Sigma_n\1_{I'}}{\sqrt{\sigma_n(I)  \sigma_n(I')}}\right)^{\frac{1}{2}}|Z_{I,I'}|,
	\end{align*}
	for some random variable $Z_{I,I'}\sim \mathcal{N}(0,1)$. Note that the system $\{Z_{I,I'} : I,I'\in \mathcal{C}_n\}$ is not necessarily independent. Let
	\[
	d_n(I, I'):=\left(2-2\frac{\1_{I}^T\Sigma_n\1_{I'}}{\sigma_n(I)}\right)^{\frac{1}{2}}.
	\]
	Since $\Sigma_n$ is a Toeplitz matrix, it follows that $\sigma_n(I) =\sigma_n(I')$ for any $I,I'\in\mathcal{C}_n$, and thus, $d_n(I,I')=d_n(I',I)$. Since $\Sigma_n$ is also positive definite, it is then easy to see that $d_n$ is a metric on $\mathcal{C}_n$.

	Now let $\mathcal{E}_n\subseteq \mathcal{C}_n$ be an $\eta_n$-net for $(\mathcal{C}_n, d_n)$, i.e. for any $I\in \mathcal{D}_n$ there is $J\in\mathcal{E}_n$, such that 
	\[
	d_n(I,J)\leq \eta_n.
	\]
	For any $I\in \mathcal{C}_n$ and $J\in\mathcal{E}_n$ we have
	\[
	\frac{\left|\1_{I}^TY\right|}{\sqrt{\sigma_n(I)}}\leq  \left|\frac{\1_{I}^TY}{\sqrt{\sigma_n(I)}} - \frac{\1_{J}^TY}{\sqrt{\sigma_n(J)}}\right| + \frac{\left|\1_{J}^TY\right|}{\sqrt{\sigma_n(J)}},
	\]
	and thus,
	\begin{align*}
	{\sup_{I\in \mathcal{C}_n}\frac{\left|\1_{I}^TY\right|}{\sqrt{\sigma_n(I)}}} &\leq \sup_{I\in\mathcal{C}_n}\inf_{J\in\mathcal{E}_n}\left|\frac{\1_{I}^TY}{\sqrt{\sigma_n(I)}}-\frac{\1_{J}^TY}{\sqrt{\sigma_n(J)}}\right|+\sup_{J\in\mathcal{E}_n}\frac{\left|\1_{J}^TY\right|}{\sqrt{\sigma_n(J)}}\\
	&= \sup_{I\in\mathcal{C}_n}\inf_{J\in\mathcal{E}_n}d_n(I,J)|Z_{I,J}| +\sup_{J\in\mathcal{E}_n}\frac{\left|\1_{J}^TY\right|}{\sqrt{\sigma_n(J)}}\\
	&\leq \eta_n\sup_{I\in\mathcal{C}_n}\inf_{J\in\mathcal{E}_n}|Z_{I,J}| +\sup_{J\in\mathcal{E}_n}\frac{\left|\1_{J}^TY\right|}{\sqrt{\sigma_n(J)}}.
	\end{align*}
	It follows that
	\begin{align*}
	{\tilde{M}_n} &\leq \eta_n\sqrt{2\log (2n)} + \sqrt{2\log\left(2 N\left(\mathcal{C}_n, d_n, \eta_n\right)\right)}\\
	&{\leq \eta_n\sqrt{2\log n} + \sqrt{2\log N\left(\mathcal{C}_n, d_n, \eta_n\right)} + (1+\eta_n)\sqrt{2\log 2}},
	\end{align*}
	where $N\left(\mathcal{C}_n, d_n, \eta_n\right)$ is the $\eta_n$-covering number of $(\mathcal{C}_n, d_n)$. Now let $I,I'\in  \mathcal{C}_n$, $I\neq I'$,  with $d_H(I,I')\leq \frac{\lambda_n}{\log n},$ where $d_H$ denotes the Hausdorff metric {on the set of subintervals of $[0,1]$ (with respect to the euclidean distance on $[0,1]$), i.e. $d_H(I,I')=|\inf I - \inf I'|$}. In addition, we assume that $\inf I < \inf I'$ without loss of generality. Then
	\begin{align*}
	\1_{I}^T\Sigma_n\1_{I'} &= (\1_{I\cap I'}^T+\1_{I\setminus I'}^T)\Sigma_n(\1_{I\cap I'}+\1_{I'\setminus I})\\
	&= \1_{I\cap I'}^T\Sigma_n\1_{I\cap I'} + \1_{I\setminus I'}^T\Sigma_n\1_{I\cap I'} + \1_{I\cap I'}^T\Sigma_n\1_{I'\setminus I} + \1_{I\setminus I'}^T\Sigma_n\1_{I'\setminus I}.
	\end{align*}
	Due to $\Sigma_n$ being symmetric and Toeplitz, we have $\1_{I\cap I'}^T\Sigma_n\1_{I'\setminus I} =\1_{I\cap I'}^T\Sigma_n\1_{I\setminus I'}$, and thus,
	\[
	\1_{I}^T\Sigma_n\1_{I'}=\1_I^T\Sigma_n\1_I - \1_{I\setminus I'}^T\Sigma_n\1_{I\setminus I'}+ \1_{I\setminus I'}^T\Sigma_n\1_{I'\setminus I}.
	\]
	It follows that 
	\begin{align*}
	d_n^2(I,I')&=2-2\left(\1_I^T\Sigma_n\1_I\right)^{-1}\left[\1_I^T\Sigma_n\1_I - \1_{I\setminus I'}^T\Sigma_n\1_{I\setminus I'}+ \1_{I\setminus I'}^T\Sigma_n\1_{I'\setminus I}\right]\\
	&= 2\left(\1_I^T\Sigma_n\1_I\right)^{-1}\left[\1_{I\setminus I'}^T\Sigma_n\1_{I\setminus I'} - \1_{I\setminus I'}^T\Sigma_n\1_{I'\setminus I}\right].
	\end{align*}

	Since $\1_{I\setminus I'}^T\Sigma_n\1_{I'\setminus I}$ is the sum over a submatrix with $r_n=n|\inf I - \inf I'|$ rows, and its lower left entry is $f_{\lfloor n\lambda_n\rfloor -1-r_n}$, we find the trivial bound
	\[
	\left|\1_{I\setminus I'}^T\Sigma_n\1_{I'\setminus I}\right|\leq \frac{n\lambda_n}{\log n}\sum_{h=\lfloor n\lambda_n(1-1/\log n)\rfloor-1}^{\infty}M(1+|h|)^{-1-\kappa} =o\left(\frac{n\lambda_n}{\log n}\right),
	\]
	as $n\to\infty$. We use Theorem 18.2.1 of \citet{ibragimov_independent_1971} to find {$\1_I^T\Sigma_n\1_I =n\lambda_n f(0)\Bigl(1+o(1)\Bigr)$} and
	\[
	\1_{I\setminus I'}^T\Sigma_n\1_{I\setminus I'}\leq f(0)\frac{n\lambda_n}{\log n}(1+o(1)),
	\]
	as $n\to\infty$. This  yields
	\[
	d_n(I,I') \leq \sqrt{\frac{2}{\log n}} + \zeta_n,
	\]
	where $\zeta_n=o\left((\log n)^{-\frac{1}{2}}\right)$. This implies that for any $I\in \mathcal{C}_n$ and for large enough $n$ we have the inclusion
	\[
	\left\{I' \in\mathcal{C}_n : d_H(I,I')\leq \frac{\lambda_n}{\log n}\right\}\subseteq \left\{I' \in\mathcal{C}_n : d_n(I,I')\leq \sqrt{\frac{2}{\log n}}+ \zeta_n\right\}.
	\]
	Thus, if we choose $\eta_n=\sqrt{\frac{2}{\log n}}+ \zeta_n$, this yields the bound
	\[
	N\left(\mathcal{C}_n, d_n, \eta_n\right)\leq \frac{\log n}{2\lambda_n},
	\]
	and consequently, 
	\[
	{\tilde{M}_n} \leq 2+\zeta_n\sqrt{2\log n} + \sqrt{2\log \frac{\log n}{2\lambda_n}}  {+ \left(1+ \zeta_n+\sqrt{\frac{2}{\log n}}\right)\sqrt{2\log 2}}.
	\]
	Thus, if we choose
	\begin{equation}\label{eq:c_alpha_n}
	c_{\alpha,n}=2+\zeta_n\sqrt{2\log n} + \sqrt{2\log\frac{2}{\alpha}}+ \sqrt{2\log \frac{\log n}{2\lambda_n}}{+\kappa_n} {+ \left(1+ \zeta_n+\sqrt{\frac{2}{\log n}}\right)\sqrt{2\log 2}} ,
	\end{equation}
	the test $\Phi_n^{\mathrm{a}}$ will have level $\alpha$ for all $n\in\N$. Note that $\zeta_n\sqrt{2\log n}=o(1)$ as $n\to\infty$. Concerning the type II error of the test $\Phi_n^{\mathrm{a}}$, recall that, under $H_1$, i.e. if $Y\sim\mathcal{N} \left(\delta \1_I,\Sigma_n\right)$ for some $\delta$ with $|\delta|>\Delta_n$, and $I\in\mathcal{I}(\lambda_n)$, we have
	\[
	\frac{\1_{I'}^TY}{\sqrt{\sigma_n(I')}}\sim\mathcal{N}\left(\frac{\delta\1^{T}_{I'}\1_I}{\sqrt{\sigma_n(I')}},1\right).
	\]
	for all $I'\in\mathcal{I}(\lambda_n)$. For $n$ large enough, it follows from plugging in 
	\eqref{eq:app_upper_bound} and \eqref{eq:c_alpha_n}, that 
	\begin{align*}
	\bar\beta(\Phi_n^{\mathrm{a}})&=\sup_{I\in\mathcal{I}(\lambda_n)}\sup_{|\delta|\ge \Delta_n}\Prob{I,\delta}{\Phi_n^{\mathrm{a}}(Y)=0}\\
	& = \sup_{I\in\mathcal{I}(\lambda_n)} \sup_{|\delta|\ge \Delta_n} \Prob{}{\sup_{I'\in\mathcal{I}(\lambda_n)}\left|Z_{I'}+\frac{\delta\1_{I'}^T \1_{I}}{\sqrt{\sigma_n(I')}}\right|\le c_{\alpha,n}}\\
	&\le \sup_{I\in\mathcal{I}(\lambda_n)} \sup_{|\delta|\ge \Delta_n} \Prob{}{\left|Z_{I}+\frac{\delta\1_{I}^T \1_{I}}{\sqrt{\sigma_n(I)}}\right|\le c_{\alpha,n}}\\
	&\le  \sup_{I\in\mathcal{I}(\lambda_n)} \sup_{|\delta|\ge \Delta_n} \Prob{}{|\delta|\frac{\1_{I}^T \1_{I}}{\sqrt{\sigma_n(I)}}-|Z_I|\le c_{\alpha,n}}\\
	&\le \Prob{}{|Z|>\Delta_n\inf_{I\in\mathcal{I}(\lambda_n)}\frac{\1_{I}^T \1_{I}}{\sqrt{\sigma_n(I)}}-c_{\alpha,n}},
	\end{align*}
	where $(Z_I)_{I\in\mathcal{I}(\lambda_n)}$ and $Z$ are (not necessarily independent) standard Gaussian random variables.
	Plugging in \eqref{eq:app_upper_bound}, we have
	\begin{multline*}
	\Delta_n\inf_{I\in \mathcal{I}(\lambda_n)}\frac{\1_{I}^T \1_{I}}{\sqrt{\sigma_n(I)}}-c_{\alpha,n}\\
	\geq \tilde  \varepsilon_n\sqrt{\log \frac{1}{\lambda_n}} -2 - \zeta_n\sqrt{2\log n} -\sqrt{2\log \frac{2}{\alpha}}- \sqrt{2 \log \frac{\log n}{2}} {-\kappa_n} {- \left(1+ \zeta_n+\sqrt{\frac{2}{\log n}}\right)\sqrt{2\log 2}},
	\end{multline*}
	for $n$ large enough. Since $\tilde \varepsilon_n\sqrt{-\log \lambda_n}{-\sqrt{\log\log n}}\to \infty$ by assumption {and $\kappa_n=o(1)$} and $\zeta_n\sqrt{2\log }=o(1)$ as $n\to\infty$, it follows that the right-hand side diverges to $\infty$. This finishes the proof.
\end{proof}

\subsubsection{Lower detection bound}\label{subsec:app_lower_bound}
We start by giving some technicalities on LR-statistics required throughout the paper at several places. As $\lambda_n$ and $\Sigma_n$ are known, the likelihood ratio $L_{\delta, I}=L_{\delta, I}(Y)$ between the distributions of $Y$ under $H_0$ and $H^{n}_{\delta, I}$ is given by
\begin{align*}
L_{I,\delta}=\exp\left[ \delta\1_I^T\Sigma_n^{-1}Y-\frac{1}{2}\delta^2\1_I^T\Sigma_n^{-1}\1_I\right].
\end{align*}
Note that, under $H_0$, the likelihood ratio $L_{\delta, I}$ follows a log-normal distribution, i.e.
\[
\log L_{I,\delta}= \delta \1_I^T\Sigma_n^{-1}Y-\frac{1}{2}\delta^2\1_I^T\Sigma_n^{-1}\1_I \stackrel{H_0}{\sim} \mathcal{N}_1\left(-\frac{1}{2}\delta^2\tilde\sigma_n(I), \delta^2\tilde\sigma_n(I)\right),
\]
where
\[
\tilde\sigma_n(I) :=\1_I^T\Sigma_n^{-1}\1_I.
\]
Note that, under $H_0$, for  $I,I'\in \mathcal{C}_n$ and $\delta\in \R$, we have $\mathbb{E}L_{I,\delta}=1$, $\var L_{I,\delta}=\exp\left(\delta^2 \tilde\sigma_n(I)\right)-1$ and $\cov(L_{I,\delta},L_{I',\delta})=\exp\left(\delta^2 \1_I^T\Sigma_n^{-1}\1_{I'}\right)-1$. Finally, let
\[
\mathcal I^0 := \left\{\left[(k-1)\lambda_n, k\lambda_n\right) : 1\leq k\leq \lfloor \lambda_n^{-1}\rfloor\right\}\subseteq \mathcal{I}(\lambda_n)
\]
be a system of non-overlapping intervals of length $\lambda_n$ as defined in \eqref{eq:I0}.

\begin{Lemma}[Lower detection bound]\label{lm:app_lower_bound}
	Fix $\alpha\in (0,1)$, and suppose that {\eqref{eq:I} holds.}
	%Assumption \ref{cond:I} 
	Let $(\Sigma_n)_{n\in\N}$ be a sequence of covariance matrices, such that  
	\begin{equation}\label{eq:lower_det_bound}
	\Delta_n\sup_{I\in\mathcal{I}(\lambda_n)}\sqrt{\tilde\sigma_n(I)}\precsim\left(\sqrt 2-\varepsilon_n\right) \sqrt{-\log \lambda_n},
	\end{equation}
	where  $(\varepsilon_n)_{n\in\N}$ is a sequence that satisfies $\varepsilon_n\to 0$ and $\varepsilon_n\sqrt{-\log \lambda_n}\to\infty$ as $n\to\infty$. In addition, assume that for some $m\in\N_0$
	\begin{equation}\label{eq:cov_cond1}
	\lim_{n\to\infty}\frac{1}{\lfloor \lambda_n^{-1}\rfloor^2}\sum_{\substack{I,I'\in\mathcal I^0 \\ n|\inf I-\inf I'|>m}} \cov\left(L_{I,\Delta_n},L_{I',\Delta_n}\right)  =0,
	\end{equation}
	as $n\to\infty$.
	
	Then any sequence of tests $\left(\Phi_n\right)_{n \in \N}$ with $\limsup_{n \to \infty} \bar\alpha \left(\Phi_n\right)\leq \alpha$ will obey $\limsup_{n \to \infty} \bar\beta \left(\Phi_n\right) \geq 1-\alpha$, i.e. the bump is asymptotically undetectable.
\end{Lemma}
\begin{proof}
	
	We employ the same strategy as in the proof of Theorem 3.1(a) of \citet{ds01}. {We bound the type II error of any given test} by an expression that does not depend on the test anymore, and then employ an appropriate $L^1$-law of large numbers for dependent arrays of random variables.
	
	For any sequence of tests $\Phi_n$ with asymptotic level $\alpha$ under $H_0$ we have
	
	\begin{align*}
	\bar \beta \left(\Phi_n\right) & = \sup_{I\in\mathcal{I}(\lambda_n)}\sup_{|\delta| \ge \Delta_n}\E{I,\delta}{1-\Phi_n(Y)} \\
	&\geq \sup_{I\in\mathcal I^0}\sup_{|\delta| \ge \Delta_n}\E{I,\delta}{1-\Phi_n(Y)} \\
	&\ge \frac{1}{\lfloor \lambda_n^{-1}\rfloor} \sum_{I\in\mathcal I^0} \sup_{|\delta| \ge \Delta_n}\E{I,\delta}{1-\Phi_n(Y)}  \\
	&\ge 1- \frac{1}{\lfloor \lambda_n^{-1}\rfloor} \sum_{I\in\mathcal I^0} \E{I,\Delta_n}{\Phi_n(Y)}\\
	& \ge  1- \frac{1}{\lfloor \lambda_n^{-1}\rfloor} \sum_{I\in\mathcal I^0} \E{0}{\Phi_n(Y) \frac{d\mathbb P_{I,\Delta_n}}{d\mathbb P_0}-\Phi_n(Y)} -\alpha  + o (1) \\
	& = 1-\E{0}{\left(~\frac{1}{\lfloor \lambda_n^{-1}\rfloor}\sum_{I\in\mathcal I^0} L_{I,\Delta_n} -1\right) \Phi_n(Y)} -\alpha + o \left(1\right) \\
	&\ge 1-\alpha-\mathbb{E}_{0}{\left| \frac{1}{\lfloor \lambda_n^{-1}\rfloor} \sum_{I\in\mathcal I^0} L_{I,\Delta_n} -1\right|}  + o \left(1\right).
	\end{align*}

	Next, we show that the array $\left\lbrace L_{\Delta_n,I} : I\in\mathcal I^0,n\in\N\right\rbrace$ is \textit{$h$-integrable with exponent 1} (recall Definition \ref{def:h_int} or see Definition 1.5 of \citet{sung08}), i.e. we show that
	\begin{equation}\label{eq:h_int_for_L}
	\sup_{n\in\N}\frac{1}{\lfloor \lambda_n^{-1}\rfloor}\sum_{I\in\mathcal I^0}\E{0}{|L_{I,\Delta_n}|}<\infty,
	\quad \text{and} \quad
	\lim_{n\to\infty}\frac{1}{\lfloor \lambda_n^{-1}\rfloor}\sum_{I\in\mathcal I^0}\E{0}{|L_{I,\Delta_n}|\1\left\lbrace |L_{I,\Delta_n}|>h(n)\right\rbrace}=0,
	\end{equation}
	where $h(n)=\lfloor \lambda_n^{-1}\rfloor^{\frac{1}{2}\left(1+\varepsilon_n\right)\left(\sqrt{2}-\varepsilon_n\right)^2}$. Since $\mathbb{E}_{0}|L_{I,\Delta_n}|=1$ for all $n\in\N$ and $I\in\mathcal I^0$, the first condition is satisfied. 
	
	Further, if $n$ large is enough, we have
	\begin{align*}
	\frac{1}{\lfloor \lambda_n^{-1}\rfloor}\sum_{I\in\mathcal I^0}\mathbb{E}_{0}\left[L_{I,\Delta_n}\1\left\lbrace L_{I,\Delta_n}> h(n)\right\rbrace\right]
	&\leq \sup_{I\in\mathcal I^0}\mathbb{E}_{0}\left[L_{I,\Delta_n}\1\left\lbrace L_{I,\Delta_n}> h(n)\right\rbrace\right]\\
	&= \sup_{I\in\mathcal I^0}\mathbb{P}\left(Z\leq\frac{\frac{1}{2}\Delta_n^2\tilde\sigma_n(I)-\log h(n)}{\Delta_n\sqrt{\tilde\sigma_n(I)}}\right)\\
	&\leq\mathbb{P}\left(Z\leq\sup_{I\in\mathcal I^0} \frac{1}{2}\Delta_n\sqrt{\tilde\sigma_n(I)} - \frac{\log h(n)}{\sup_{I\in\mathcal I^0} \Delta_n\sqrt{\tilde\sigma_n(I)}}\right)\\
	&\stackrel{\text{(a)}}{\leq} \mathbb{P}\left(Z\leq -\varepsilon_n(\sqrt{2}-\varepsilon_n)\sqrt{-\log \lambda_n}\right),
	\end{align*}
	where $Z$ is a standard Gaussian random variable.  The inequality (a) follows immediately from \eqref{eq:lower_det_bound} and the definition of $h(n)$. The claim follows from the assumption that $\lim_{n\to\infty}\varepsilon_n\sqrt{-\log \lambda_n}=\infty$ as $n\to\infty$.\\	
	Then, given that \eqref{eq:cov_cond1} and \eqref{eq:h_int_for_L} hold, it follows from an $L^1$-law of large numbers for dependent arrays (recall Theorem \ref{th:WLLN} or see Theorem 3.2 of \citet{wanghu14}), that
	\begin{equation}\label{eq:L1_LLN}
	\mathbb{E}_{0}{\left|\frac{1}{\lfloor \lambda_n^{-1}\rfloor} \sum_{I\in\mathcal I^0} L_{\Delta_n,I} -1\right|}\to 0,
	\end{equation}
	as $n\to\infty$, which finishes the proof.
\end{proof}

\subsubsection{Proof of Theorem \ref{thm:mainthm(poly)}}\label{subsec:2.1proof}

In the setting described in Theorem \ref{thm:mainthm(poly)} the noise vector $\xi_n$ in model \eqref{eq:model_dep} is given by $n$ consecutive realizations of a stationary centered Gaussian process with the square summable autocovariance function $\gamma(h)$, $h\in\mathbb{Z}$ and the spectral density $f$. We suppose that Assumption  \ref{cond:P} is satisfied, i.e. the autocovariance of $\xi_n$ has a  polynomial decay. In terms of $\Sigma_n$, this means
\[
|\Sigma_n(i,j)|\leq C(1+|i-j|)^{-(1+\kappa)},
\]
for $1\leq i,j\leq n$ and some constants $C>0$ and $\kappa>0$.

In order to apply {Lemma \ref{lm:app_lower_bound}} in such a setting, first, we need to examine the asymptotic behavior of the {coefficients $\tilde \sigma_n \left(I\right)$}, and second, we need to verify that condition \eqref{eq:cov_cond1} is satisfied under the lower detection boundary condition \eqref{eq:lower_det_bound} and Assumption \ref{cond:P}.

For the setting of Theorem \ref{thm:mainthm(poly)}, we will do the former in Lemma \ref{lm:arma_beta_nk} and the latter in Lemma \ref{lm:arma_cov_cond}.

\begin{Lemma}\label{lm:arma_beta_nk}
	If Assumption \ref{cond:P} holds, then for any $I\in\mathcal{I}(\lambda_n)$, it follows that
	\[
	\tilde\sigma_n(I)= \frac{n\lambda_n}{f(0)}(1+o(1)),
	\]
	as $n\to\infty$.
\end{Lemma}

\begin{proof}
	We are inspired by the proof of Proposition C.1 in \citet{Keshavarz2018arXiv}, that was dropped from the final paper \cite{Keshavarz2018}, although we are able to make some simplifications, since a slightly weaker result suffices for our purposes.
	In addition, we use this opportunity to fix several minor inaccuracies in their proof.\\
	Recall that $\mathcal{T}(f)$ is the infinite Toeplitz matrix generated by the spectral density $f$ and  that $\Sigma_n=\mathcal{T}_n(f)$ is the corresponding truncated Toeplitz matrix.
	
	Let $\mathcal{T}(g)$ be the infinite Toeplitz matrix generated by $g=1/f$, i.e. the matrix with elements $\mathcal{T}(g)(i,j)=g_{|i-j|}$, where $g_0, g_1,\ldots$ are the Fourier coefficients of $g$. Let $\mathcal{H}(f)$ and $\mathcal{H}(g)$ be the Hankel matrices generated by $f$ and $g$, respectively, i.e. the matrices
	\[
	\mathcal{H}(f)=\begin{pmatrix}
	f_1 & f_2 & f_3 & \dots  \\ 
	f_2 &  f_3 & f_4 & \dots \\ 
	f_3  & f_4 & f_5 & \dots \\ 
	\vdots & \vdots & \vdots & \ddots
	\end{pmatrix} \quad \mbox{and}\quad 
	\mathcal{H}(g)=\begin{pmatrix}
	g_1 & g_2 & g_3 & \dots  \\ 
	g_2 &  g_3 & g_4 & \dots \\ 
	g_3  & g_4 & g_5 & \dots \\ 
	\vdots & \vdots & \vdots & \ddots
	\end{pmatrix} 
	\]
	It follows from Proposition 1.12 of \citet{bottcher2012}, that
	\[
	\mathcal{T}(f)^{-1}=\mathcal{T}(g)+\mathcal{T}(f)^{-1}\mathcal{H}(f)\mathcal{H}(g).
	\]
	Let $v_I$ be the extension of the vector $\1_I$ to an element of $l^2$ by zero-padding. As in \citet{Keshavarz2018}, from the above identity and the definition of the operator norm, we find 
	\begin{align*}
	\left| v_I^T\mathcal{T}(f)^{-1}v_I -  v_I^T\mathcal{T}(g)v_I \right|
	&=\left|\left\langle \mathcal{H}(f)\mathcal{T}(f)^{-1}v_I,\mathcal{H}(g)v_I\right\rangle\right| \\
	&\leq \Vert \mathcal{H}(f)\mathcal{T}(f)^{-1}\Vert \Vert v_I\Vert_{\ell_2}\Vert \mathcal{H}(g) v_I\Vert_{\ell_2}\\
	&\leq\Vert \mathcal{H}(f)\mathcal{T}(f)^{-1}\Vert \sqrt{n\lambda_n} \left[\sum_{\{r:v_I(r)=1\}} \left\Vert \mathcal{H}(g)e_r\right\Vert_{\ell_2}\right]
	\end{align*}
	where $e_r=(0,\ldots, 0,1,0,\ldots)^T$ is the sequence whose $r$-th entry is 1, and $\Vert\mathcal{H}(f) \mathcal{T}(f)^{-1}\Vert$ is the operator norm of $\mathcal{H}(f)\mathcal{T}(f)^{-1}$ as an operator from $\ell^2$ to $\ell^2$. Since $\Vert \mathcal{T}(f)\Vert = \sup_{\nu\in [0,1)}f(\nu)<\infty$, we have $\Vert \mathcal{T}(f)^{-1}\Vert<\infty$ by the inverse mapping theorem. It follows that $\Vert \mathcal{H}(f)\mathcal{T}(f)^{-1}\Vert<\infty$, because clearly $\Vert \mathcal{H}(f)\Vert<\infty$. Since $f$ is bounded away from $0$, it is well known that the Fourier coefficients $g_k$, $k\in\mathbb{Z}$ of $g$ decay at the same rate as the Fourier coefficients of $f$, i.e.
	\[
	|g_k| \leq C'(1+|k|)^{-(1+\kappa)},
	\]  
	for $k\in \mathbb{Z}$. Following \citet{Keshavarz2018}) we see that
	\begin{align*}
	\sum_{\{r:v_I(r)=1\}}\Vert \mathcal{H}(g) e_ {r}\Vert_{\ell_2} &= \sum_{\{r:v_I(r)=1\}}\left(\sum_{j=r}^{\infty} |g_j|^2\right)^{\frac{1}{2}} \leq \sum_{\{r:v_I(r)=1\}}\left(\int_r^\infty x^{-2(1+\kappa)}dx\right)^{\frac{1}{2}}\\
	&\leq C''\sum_{\{r:v_I(r)=1\}} r^{-\left(\frac{1}{2}+\kappa\right)} \leq C''\sum_{r=1}^{\lfloor n\lambda_n\rfloor} r^{-\left(\frac{1}{2}+\kappa\right)}.
	\end{align*}
	It is then easy to see that the last expression is $O\left((n\lambda_n)^{\frac{1}{2}-\kappa}\right)$ if $\kappa<\frac{1}{2}$, and $O\left(\log (n\lambda_n)\right)$ if $\kappa=\frac{1}{2}$. Lastly, it is also clearly bounded if $\kappa>\frac{1}{2}$. Hence, in any of these cases it holds that
	\[
	\sum_{\{r:v_I(r)=1\}}\Vert \mathcal{H}(g) e_ {r}\Vert_{\ell_2}=o\left(\sqrt{n\lambda_n}\right).
	\]
	Thus,
	\[
	\left|v_I^T\mathcal{T}(f)^{-1}v_I-v_I^T\mathcal{T}(g)v_I \right|=o\left(n\lambda_n\right).
	\]
	We now need to bound $v_I^T\mathcal{T}(g) v_I$. Let $(X_t)_{t\in\N}$ be a stationary random process with spectral density $g$. Then 
	\[
	v_I^T\mathcal{T}(g) v_I=\var\left(\sum_{\{t: v_I(t)=1\}} X_t\right) = n\lambda_n(g(0)+o(1)),
	\]
	as $n\to\infty$, where the last equality is due to Theorem 18.2.1 of \citet{ibragimov_independent_1971}, see Section 5.3 of the Appendix for the precise statement of the theorem. (Note that $g$ is continuous at $0$ and $g(0)>0$.) Thus, 
	\[
	v_I^T\mathcal{T}(f)^{-1}v_I= n\lambda_n(g(0)+o(1)).
	\]
	Finally, by Theorem 2.11 of \citet{bottcher2000}, we have
	\[
	\tilde\sigma_n(I)=v_I^T\Sigma_n^{-1}v_I=v_I^T\mathcal{T}(f)^{-1}v_I+\tilde{v}_I^T\left[\mathcal{T}(f)^{-1}-\mathcal{T}(g)\right]\tilde{v}_I+ v_I^TD_n v_I,
	\]
	where $\Vert D_n\Vert\to 0$, as $n\to\infty$, and $\tilde{v}_I$ arises from $v_I$ through the transformation
	\[
	\tilde{v}_I=\left(v_I(n), \ldots, v_I(1),0,0,\ldots\right).
	\]
	As above, we have
	\[
	\left|\tilde{v}_I^T\left[\mathcal{T}(f)^{-1}-\mathcal{T}(g)\right]\tilde{v}_I\right|=o\left(n\lambda_n\right),
	\]
	and clearly, by Cauchy-Schwarz,
	\[
	\left|v_I^TD_n v_I\right|\leq \Vert v_I\Vert^2\Vert D_n\Vert=o\left(n\lambda_n\right).
	\]
	This concludes the proof.
	
\end{proof}

\begin{Lemma}\label{lm:arma_cov_cond}
	If Assumption \ref{cond:P} holds, and given that
	\begin{equation}\label{eq:lowerboundcond}
	\Delta_n\sup_{I\in\mathcal{I}(\lambda_n)}\sqrt{\tilde\sigma_n(I)}\precsim (\sqrt{2}-\varepsilon_n)\sqrt{-\log \lambda_n}
	\end{equation}
	for a sequence $(\varepsilon_n)_{n\in\N}$ satisfying $\varepsilon_n\to 0$ and $\varepsilon_n\sqrt{-\log \lambda_n}\to\infty$ as $n\to\infty$, then condition \eqref{eq:cov_cond1} holds with $m=1$, i.e.
	\[
	\lim_{n\to\infty}\lambda_n^2\sum_{\substack{I,I'\in\mathcal I^0 \\ n|\inf I-\inf I'|>1}} \exp\left(\Delta_n^2\1_I^T \Sigma_n^{-1}\1_{I'}\right)-1 =0,
	\]
\end{Lemma}

\begin{proof}
	For $I,I'\in\mathcal I^0$ with $n|\inf I-\inf I'|>1$. Write
	\begin{align*}
	\exp\left(\Delta_n^2\1_I^T \Sigma_n^{-1}\1_{I'}\right)-1 &= \sum_{p=1}^{\infty}\frac{1}{p!}\left[\Delta_n^2\1_I^T\Sigma_n^{-1}\1_{I'}\right]^p\\
	&= \sum_{p=1}^{\infty}\frac{1}{p!}\left[\frac{1}{2}\Delta_n^2\sqrt{\tilde\sigma_n(I)\tilde\sigma_n(I')}\right]^p\left[2\frac{\1_I^T\Sigma_n^{-1}\1_{I'}}{\sqrt{\tilde\sigma_n(I)\tilde\sigma_n(I')}}\right]^p.
	\end{align*}
	If $n\lambda_n$ is an integer, the latter term $\1_I^T\Sigma_n^{-1}\1_{I'}$ is the sum over a square submatrix of $\Sigma_n^{-1}$, and if $n\lambda_n$ is not an integer, then the number of non-zero entries of $\1_I$ and $\1_{I'}$ cannot differ by more than $1$. From Lemma A.1 of \citep{hj10} (see also Section \ref{app:precision_matrix} in the appendix), it trivially follows that
	\begin{align*}
	\left|\1_I^T\Sigma_n^{-1}\1_{I'}\right| &\leq  C' \left\lceil n\lambda_n\right\rceil\sum_{t=1}^{\left\lceil n\lambda_n\right\rceil}\left(n|\inf I-\inf I'|\left\lfloor n\lambda_n\right\rfloor+t\right)^{-(1+\kappa)}\\
	&\leq  C' \left\lceil n\lambda_n\right\rceil\sum_{t=1}^{\left\lceil n\lambda_n \right\rceil}\left(\left\lfloor n\lambda_n \right\rfloor\right)^{-(1+\kappa)}=o\left(n\lambda_n\right).
	\end{align*}
	
	From Lemma \ref{lm:arma_beta_nk}, we know that $\sqrt{\tilde\sigma_n(I)\tilde\sigma_n(I')}=\frac{n\lambda_n}{f(0)}(1+o(1))$ as $n\to\infty$, and thus, it follows that $\sqrt{\tilde\sigma_n(I)\tilde\sigma_n(I')}^{-1}\1_I^T\Sigma_n^{-1}\1_{I'}\to 0$ as $n\to\infty$. Hence, for $n$ large enough, we have
	\begin{multline}\label{eq:proof:arma_cov_cond}
	\left| \sum_{p=1}^{\infty}\frac{1}{p!}\left[\frac{1}{2}\Delta_n^2\sqrt{\tilde\sigma_n(I)\tilde\sigma_n(I')}\right]^p\left[2\frac{\1_I^T\Sigma_n^{-1}\1_{I'}}{\sqrt{\tilde\sigma_n(I)\tilde\sigma_n(I')}}\right]^p\right|\\ \leq 2\frac{\left|\1_I^T\Sigma_n^{-1}\1_{I'}\right|}{\sqrt{\tilde\sigma_n(I)\tilde\sigma_n(I')}}\exp\left[\frac{1}{2}\Delta_n^2\sqrt{\tilde\sigma_n(I)\tilde\sigma_n(I')}\right].
	\end{multline}
	
	Note that from the lower detection boundary condition \eqref{eq:lowerboundcond} it immediately follows that
	\begin{equation}\label{eq:proof:arma_cov_cond_2}
	\exp\left[\frac{1}{2}\Delta_n^2\sqrt{\tilde\sigma_n(I)\tilde\sigma_n(I')}\right]\leq \lambda_n^{-\frac{1}{2}(\sqrt{2}-\varepsilon_n)^2}\leq \lambda_n^{-1}
	\end{equation}
	for $n$ large enough. Applying Lemma A.1 of \citep{hj10} again, it follows that $|\Sigma_n^{-1}(i,j)|\leq C(1+|i-j|)^{-(1+\kappa)}$  for some $C>0$. Let $\Phi_n$ be the $n\times n$-matrix with entries $\Phi_n(i,j)=C(1+|i-j|)^{-(1+\kappa)}$, and let $\Phi(\nu)=\sum_{h=-\infty}^{\infty}C(1+|i-j|)^{-(1+\kappa)}e^{-2\pi i h\nu}$. Then
	\begin{align}\label{eq:proof:arma_cov_cond_3}
	\sum_{\substack{I,I'\in\mathcal I^0 \\ n|\inf I-\inf I'|>1}} |\1_I^T\Sigma_n^{-1}\1_{I'}|\leq \sum_{\substack{I,I'\in\mathcal I^0 \\ I\neq I'}} \1_I^T\Phi_n\1_{I'} \leq  \sum_{i,j=1}^n \Phi_n(i,j)-\sum_{I\in\mathcal I^0} \1_I^T\Phi_n\1_I\stackrel{\text{(a)}}{=}o(n),
	\end{align}
	where (a) follows from Theorem 18.2.1 of \citet{ibragimov_independent_1971}, since it yields that $\sum_{i,j=1}^n \Phi_n(i,j)= n\Phi(0)+o(n)$ and $\1_I^T\Phi_n\1_I=n\lambda_n\Phi(0)+o\left(n\lambda_n\right)$ for any $I\in\mathcal I^0$.
	
	Finally, combining \eqref{eq:proof:arma_cov_cond}, \eqref{eq:proof:arma_cov_cond_2} and \eqref{eq:proof:arma_cov_cond_3}, and once again using that $\sqrt{\tilde\sigma_n(I)\tilde\sigma_n(I')}=\frac{n\lambda_n}{f(0)}(1+o(1))$ as $n\to\infty$, we find
	\[
	\sum_{\substack{I,I'\in\mathcal I^0 \\ n|\inf I-\inf I'|>1}}2\frac{\left|\1_I^T\Sigma_n^{-1}\1_{I'}\right|}{\sqrt{\tilde\sigma_n(I)\tilde\sigma_n(I')}}\exp\left[\frac{1}{2}\Delta_n^2\sqrt{\tilde\sigma_n(I)\tilde\sigma_n(I')}\right] = o\left(\frac{1}{\lambda_n^2}\right),
	\]
	which concludes the proof.
\end{proof}

Since Lemma \ref{lm:arma_cov_cond} guarantees that Lemma \ref{lm:app_lower_bound} can be applied in the setting of Theorem \ref{thm:mainthm(poly)}, the proof of the latter now follows immediately from Lemmas \ref{lm:app_upper_bound} and \ref{lm:app_lower_bound}.

\begin{proof}[Proof of Theorem \ref{thm:mainthm(poly)}]
	The two Lemmas \ref{lm:app_upper_bound} and \ref{lm:app_lower_bound} yield that the asymptotic detection boundary is (in terms of $\Delta_n$) given by
	\begin{equation}\label{eq:det_bound_prelim}
	(\sqrt{2}-\varepsilon_n)\sqrt{\frac{-\log\lambda_n}{n\lambda_n}}\sup_{I\in\mathcal{I}(\lambda_n)} \sqrt{\frac{n\lambda_n}{\tilde\sigma_n(I)}}\precsim\Delta_n\precsim (\sqrt{2}+\tilde\varepsilon_n)\sqrt{\frac{-\log\lambda_n}{n\lambda_n}}\inf_{I\in\mathcal{I}(\lambda_n)}\sqrt{\frac{\sigma_n(I)}{n\lambda_n}},
	\end{equation}
	as $n\to\infty$. For any $I\in\mathcal{I}(\lambda_n)$, it follows from Theorem 18.2.1 of \citet{ibragimov_independent_1971} that 
	\[
	\sigma_n(I)=n\lambda_n f(0)(1+o(1)),
	\]
	and Lemma \ref{lm:arma_beta_nk} yields
	\[
	\tilde\sigma_n(I)=\frac{n\lambda_n}{f(0)}(1+o(1)),
	\]
	as $n\to\infty$. Plugging this into \eqref{eq:det_bound_prelim} finishes the proof.
\end{proof}

\subsubsection{Remaining proofs}

\begin{proof}[Proof of Theorem \ref{thm:general:nonoverlap}]
	Note that for any $1\leq k\leq \lfloor \lambda_n^{-1}\rfloor$, under $H_0$, the random variables {$\frac{\1_{I_k}^T\Sigma_n^{-1}Y}{\sqrt{\tilde\sigma_k}}$ with $\tilde\sigma_k$ as in \eqref{eq:sigmak}} are identically distributed (dependent) standard Gaussian. Note that $\tilde \sigma_k = \tilde \sigma_n \left(I_k\right)$ with our former notation. The union bound and the elementary tail inequality $\Prob{}{|Z|>x}\le 2e^{-x^2/2}$ for  $Z\sim \mathcal N(0,1)$, yield
	\begin{align*}
	\tilde \alpha(\Phi_n^{\mathrm{d}})&= \Prob{0}{T_n(Y)>c_{\alpha,n}}\\
	&\leq \lfloor\lambda_n^{-1}\rfloor \sup_{1\leq k\leq  \lfloor\lambda_n^{-1}\rfloor} \Prob{0}{\frac{|\1_{I_k}^T\Sigma_n^{-1}Y|}{\sqrt{\tilde\sigma_k}}>c_{\alpha,n}} \\ 
	&=  \lfloor\lambda_n^{-1}\rfloor  \mathbb{P}\left[|Z|>c_{\alpha,n}\right] \leq 2  \lfloor\lambda_n^{-1}\rfloor \exp\left(-\frac{c^2_{\alpha,n}}{2}\right)\le \alpha.
	\end{align*}
	
	This proves $\tilde\alpha \left(\Phi_n^{\mathrm{d}}\right)\leq \alpha$ for all $n \in \mathbb N$. 
	
	Concerning the type II error, note that, under $H_1$, i.e. if $Y\sim\mathcal{N} \left(\delta_n \1_{I_k},\Sigma_n\right)$ for some $k\in\lbrace1,\ldots,\lfloor\lambda_n^{-1}\rfloor\rbrace$, we have for all local test statistics on the right-hand side of \eqref{eq:LRT_nonover} that
	\[
	\frac{\1_{I_m}^T\Sigma_n^{-1}Y}{\sqrt{\sigma_m}}\sim\mathcal{N}\left(\frac{\delta_n\1_{I_m}\Sigma_n^{-1}\1_{I_k}}{\sqrt{\sigma_m}},1\right),\quad m=1,\dots,\lfloor\lambda_n^{-1}\rfloor.
	\]
	Plugging in \eqref{eq:test_thresh} and \eqref{eq:upper_det_bound}, it follows that the type II error satisfies
	\begin{align*}
	\tilde\beta(\Phi_n^{\mathrm{d}},\Sigma_n,\Delta_n,\lambda_n)&=\sup_{1\leq k\leq \lfloor \lambda_n^{-1}\rfloor}\sup_{|\delta_n|\ge \Delta_n}\Prob{\delta_n,k}{\Phi_n^{\mathrm{d}}(Y)=0}\\
	& = \sup_{1\leq k\leq \lfloor\lambda_n^{-1}\rfloor} \sup_{|\delta_n|\ge \Delta_n} \Prob{}{\sup_{1\leq m\leq \lfloor \lambda_n^{-1}\rfloor}\left|Z_m+\frac{\delta_n\1_{I_k} \Sigma_n^{-1}\1_{I_m}}{\sqrt{\sigma_m}}\right|\le c_{\alpha,n}}\\
	&\le \sup_{1\leq k\leq \lfloor\lambda_n^{-1}\rfloor} \sup_{|\delta_n|\ge \Delta_n} \Prob{}{\left|Z_{k}+\delta_n\sqrt{\tilde\sigma_k}\right|\le c_{\alpha,n}}\\
	&\le  \sup_{1\leq k\leq \lfloor\lambda_n^{-1}\rfloor} \sup_{|\delta_n|\ge \Delta_n} \Prob{}{|\delta_n|\sqrt{\tilde\sigma_k}-|Z_k|\le c_{\alpha,n}}\\
	&\le \Prob{}{|Z|>\Delta_n\inf_{1\leq k\leq \lfloor\lambda_n^{-1}\rfloor}\sqrt{\tilde\sigma_k}-c_{\alpha,n}}
	\end{align*}
	which proves the claim.
\end{proof}

\begin{proof}[Proof of Corollary \ref{cor:disjoint} ]
	The claim follows directly from Theorem \ref{thm:general:nonoverlap} and the standard Gaussian tail bound $\Prob{}{|Z|>x}\le 2e^{-x^2/2}$ via
	\begin{align*}
	&\Prob{}{|Z|>\Delta_n\inf_{1\leq k\leq \lfloor\lambda_n^{-1}\rfloor}\sqrt{\tilde\sigma_k}-c_{\alpha,n}}\\
	=& \Prob{}{|Z|>(1+\varepsilon_n)\sqrt{-2\log \lambda_n}-\sqrt{-2\log (\lambda_n)+2\log (2\alpha^{-1})}}\\
	\le& \Prob{}{|Z|>\varepsilon_n\sqrt{-2\log \lambda_n} -\sqrt{2\log (2\alpha^{-1})}}\\
	\le& \exp\left(-\frac12 \left(\varepsilon_n\sqrt{-2\log \lambda_n} -\sqrt{2\log (2\alpha^{-1})}\right)^2\right)\\
	\le& \exp\left(-\frac12 \left(\sqrt{2\log (\alpha^{-1})} \right)^2\right)=\alpha.
	\end{align*}
\end{proof}

\subsection{Proofs for Section \ref{sec:arma_finite}}

\begin{proof}[Proof of Theorem \ref{thm:arma_result}]
	It is well-known (see, for example, \cite{shumway2000}, Sections 3.3--3.4), that the autocovariance function $\gamma$ of an ARMA process is exponentially decaying, i.e.
	\[
	|\Sigma_n(i,j)|=|\gamma(i-j)|\leq Ce^{-\kappa |i-j|},
	\]
	for some $C>0$, some $\kappa>0$  and all $1\leq i,j\leq n$. Thus, Assumption \ref{cond:P} is satisfied, and Theorem \ref{thm:arma_result} follows immediately from Theorem \ref{thm:mainthm(poly)}.
\end{proof}

\subsubsection{Properties of the precision matrix of an AR($p$) process}\label{subsec:ARp_precision_matrix}

Let $Z_t$ be a stationary AR($p$) process defined in~(\ref{eq:ARp}) and $\Sigma_n$ be the covariance matrix of $n$ consecutive realizations of $Z_t$.  Then the precision matrix $\Sigma_n^{-1}=(\Sigma_n^{-1}(i,j))$, $i,j=1,\dots,n$ is a $n\times n$ symmetric $2p+1$-diagonal matrix  with the upper-triangle elements given by (see \cite{siddiqui1958})
\begin{equation}\label{eq:Sigmainv}
\Sigma_n^{-1}(i,j)= \begin{cases}
\sum\limits_{t=0}^{i-1} \phi_{t}\phi_{t+j-i}, & 1\le i\le j\le p\\
\sum\limits_{t=0}^{p+i-j} \phi_{t}\phi_{t+j-i},& 1\le i\le n-p, \ \max(i,p+1)\le j\le i+p\\
\sum\limits_{t=0}^{n-j}\phi_{t}\phi_{t+j-i},& n-p+1\le i\le j\ge n\\
0,& i+p<j\le n,\ i\le n-p.
\end{cases}
\end{equation}
Note that $\Sigma_n^{-1}$ is symmetric with respect to both the main diagonal and the antidiagonal, so that $\Sigma_n^{-1}(i,j)=\Sigma_n^{-1}(j,i)$ and $\Sigma_n^{-1}(i,j)=\Sigma_n^{-1}(n+1-j,n+1-i)$. 

We can see from~(\ref{eq:Sigmainv}) that $\Sigma_n^{-1}$ has two symmetric blocks $L=(l_{ij})$ and $R=(r_{ij})$ of size $p$ with the elements related as $l_{ij}=r_{p+1-i,p+1-j}=\Sigma_n^{-1}(i,j)$, $i,j=1,\dots,p$  (red blocks in Fig.~\ref{fig:Sigmainv}). The other elements of $\Sigma_n^{-1}$ are constant on the diagonals and are given by $\Sigma_n^{-1}(i,i+k)=D_k$, $i=p-k+1$, $k=1,\dots,p$ (blue parts of the matrix in  Fig.~\ref{fig:Sigmainv}), where 
$$
D_k=\sum\limits_{t=0}^{p-k} \phi_{t}\phi_{t+k},\quad k=0,\dots,p.
$$

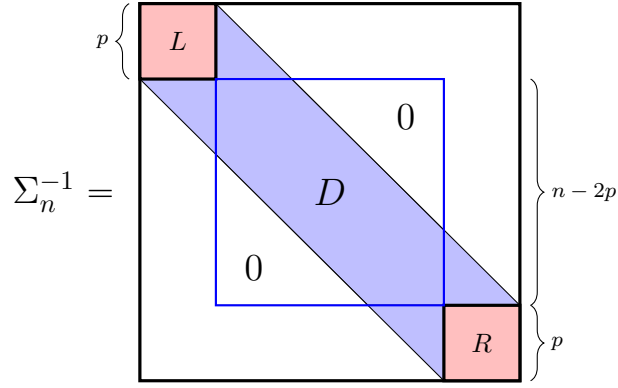
\begin{figure}[htbp!]\label{fig:Sigmainv}
	\centering
	\begin{tikzpicture} 
	\draw[very thick] (0,0)--(0,5)--(5,5)--(5,0)--cycle;
	\draw [black,fill=blue!25] (1,5)--(2,4)--(1,4)--cycle;
	\draw [black,fill=blue!25] (0,4)--(1,4)--(1,3)--cycle;
	\draw [black,fill=blue!25] (1,3)--(1,4)--(2,4)--(4,2)--(4,1)--(3,1)--cycle;
	\draw [black,fill=white] (1,1)--(1,3)--(3,1)--cycle;
	\draw [black,fill=white] (2,4)--(4,4)--(4,2)--cycle;
	\draw [black,fill=blue!25] (4,0)--(4,1)--(3,1)--cycle;
	\draw [black,fill=blue!25] (4,1)--(4,2)--(5,1)--cycle;
	\draw [very thick,color=black, fill=red!25] (0,4)--(0,5)--(1,5)--(1,4)--cycle;
	\draw [very thick,color=black, fill=red!25] (5,1)--(5,0)--(4,0)--(4,1)--cycle;
	\draw[thick,color=blue] (1,1)--(1,4)--(4,4)--(4,1)--cycle;
	\draw (0.5,4.5) node{$L$};
	\draw (4.5,0.5) node{$R$};
	\draw (2.5,2.5) node{\Large $D$};
	\draw (3.5,3.5) node{\Large $0$};
	\draw (1.5,1.5) node{\Large $0$};
	\draw (-1,2.5) node{\Large $\Sigma_n^{-1}=$};
	\draw [decorate,decoration={brace,amplitude=5pt},xshift=-4pt,yshift=0pt]
	(0,4) -- (0,5) node [black,midway,xshift=-10pt] 
	{\footnotesize $p$};
	\draw [decorate,decoration={brace,mirror,amplitude=5pt},xshift=4pt,yshift=0pt]
	(5,1) -- (5,4) node [black,midway,xshift=20pt] 
	{\footnotesize $n-2p$};
	\draw [decorate,decoration={brace,mirror,amplitude=5pt},xshift=4pt,yshift=0pt]
	(5,0) -- (5,1) node [black,midway,xshift=10pt] 
	{\footnotesize $p$};
	\end{tikzpicture}
	\caption{The matrix $\Sigma_n^{-1}$ is symmetric $2p+1$-diagonal, the blocks $L$ and $R$ of size $p$ are of size $p$, the blue part is has the same values $D_k$ on the diagonals. The white part consists of zeros. }
\end{figure}

We are interested in the diagonal block sums of $\Sigma_n^{-1}$ over the blocks of size $r$. We suppose that $1\le r< \lfloor n/2\rfloor -p$. The block sums of interest are 
\begin{equation}\label{Srk}
S_{r,m} = \1_{r,m}^T \Sigma_n^{-1}  \1_{r,m},\quad m=1,\dots,n-r+1
\end{equation}
where $ \1_{r,m}\in \mathbb R^n$ is the vector with entries
$$
\1_{r,m}(i)=\begin{cases}
1,& i=m,\dots,m+r-1,\\
0,&\mbox{otherwise}
\end{cases}
$$
Note that the key quantities $\tilde\sigma_k$ that appear in the lower and upper bounds of testing~(\ref{eq:lower_det_bound}) and (\ref{eq:upper_det_bound}) are related to~(\ref{Srk}) as follows,
$$
\tilde\sigma_k= S_{\lfloor n\lambda_n\rfloor ,(k-1)\lfloor n\lambda_n\rfloor+1},\quad k=1,\dots,\lfloor \lambda_n^{-1}\rfloor. 
$$

\begin{Lemma}\label{lem:SB}
	Suppose that $1\le r\le  n -2p$ and that $n\ge 3p$. The quantities $S_{r,m}$, $m=1,\dots,n-r+1$ can be calculated directly using the following recursive formulas.
	\begin{enumerate}
		\item The first block sum is given by 
		\begin{equation}\label{SB1r}
		S_{r,1}  = \begin{cases}\sum\limits_{j=1}^{r} \left(\sum\limits_{t=0}^{j-1} \phi_t\right)^2,& 1\le r \le p,\\
		\sum\limits_{j=1}^p \left(\sum\limits_{t=0}^{j-1} \phi_t\right)^2 + (r-p)\left(\sum\limits_{t=0}^{p} \phi_t\right)^2,& p\le r\le n-p
		\end{cases}
		\end{equation}
		\item If $r\le p$, then 
		\begin{equation}\label{SBmr}
		S_{r,m+1}=S_{r,m} + \begin{cases}
		\left(\sum\limits_{t=0}^{r-1} \phi_{t+i}\right)^2,& 1\le m\le p+1-r\\
		\left(\sum\limits_{t=0}^{p-i} \phi_{t+i}\right)^2,& p+1-r \le m\le p\\
		0,& p+1\le m\le n-p-r\\
		-\left(\sum\limits_{t=n-i-p}^{r-1}\phi_{n-i-t} \right)^2,&  n-p-r+1\le m\le n-p\\
		-\left(\sum\limits_{t=0}^{r-1} \phi_{n-i-t} \right)^2,&  n-p\le m\le n-r\\
		\end{cases}
		\end{equation}
		\item If $p\le r \le n - 2p$, then 
		\begin{equation}\label{SBmr_bigr}
		S_{r,m+1}=S_{r,m}+\begin{cases}
		\left(\sum\limits_{t=0}^{p-i} \phi_{t+i}\right)^2,& 1\le m\le p\\
		0,& p+1\le im\le n-p-r\\
		-\left(\sum\limits_{t=n-i-p}^{r-1} \phi_{n-i-t} \right)^2,&  n-p-r+1\le m\le n-r.
		\end{cases}
		\end{equation}
		
	\end{enumerate}
\end{Lemma}

The proof of the lemma is omitted. It follows from simple algebra and the relation
$$
D_0+ 2\sum_{k=1}^{p} D_k=\sum\limits_{t=0}^{p} \phi_{t}^2 +2 \sum_{k=1}^{p} \sum\limits_{t=0}^{p-k} \phi_{t}\phi_{t+k} = \left(\sum_{t=0}^p \phi_t\right)^2. 
$$

Using the result of Lemma~\ref{lem:SB}, we can calculate the constants $\tilde\sigma_k$. 
\begin{proof}[Proof of Lemma \ref{lm:finite_beta}]
	According to definition~\ref{Srk}, the quantities $\tilde\sigma_k$ can be written as 
	$$
	\tilde\sigma_k= S_{\lfloor n\lambda_n\rfloor ,(k-1)\lfloor n\lambda_n\rfloor+1},\quad k=1,\dots,\lfloor \lambda_n^{-1}\rfloor. 
	$$
	Note that it follows immediately from Lemma~\ref{lem:SB} that for any fixed $1\le r\le n-2p$ the function $S_{r,m}$, $m=1,\dots,n-r+1$ is  monotone increasing for $m\le p+1$, constant for $p+1\le m\le n-p-r+1$ and decreasing for $m\ge n-p-r+1$. Moreover, this function is symmetric in a sense that $S_{r,m}=S_{r,n-r-m+2}$, $m=1,\dots,n-r+1$. Therefore, it follows that
	\[
	\inf_{1\le k\le \lfloor \lambda_n^{-1}\rfloor}\tilde\sigma_k = \min_{1\le k\le \lfloor \lambda_n^{-1}\rfloor} S_{\lfloor n\lambda_n\rfloor ,(k-1)\lfloor n\lambda_n\rfloor+1} =S_{\lfloor n\lambda_n\rfloor ,1}
	\]
	and
	\[ 
	\sup_{1\le k\le \lfloor \lambda_n^{-1}\rfloor}\sigma_{k} = \max_{1\le k\le \lfloor \lambda_n^{-1}\rfloor} S_{\lfloor n\lambda_n\rfloor ,(k-1)\lfloor n\lambda_n\rfloor+1}  = S_{\lfloor n\lambda_n\rfloor,p+1}.
	\]
	Note that the condition $\lfloor n\lambda_n\rfloor<n-2p$ will guarantee that the maximum is attained at the interval where the function $S$ is constant (for some $k$ that satisfies $p+1\le (k-1)\lfloor n\lambda_n\rfloor+1\le n-p-r+1$) and, consequently, will be equal to $S_{\lfloor n\lambda_n\rfloor,p+1}$. 
	
	We obtain the statement of the lemma applying the recursive formulas of Lemma~\ref{lem:SB}.
\end{proof}

\section*{Acknowledgments}

Axel Munk and Frank Werner gratefully acknowledge financial support by the German Research Foundation DFG through subproject A07 of CRC 755, and Markus Pohlmann acknowledges support through RTG 2088. We are furthermore grateful to helpful comments of two anonymous referees and the associate editor.

\bibliographystyle{apalike}
\bibliography{literature}

\end{document}